\documentclass{amsart}

\usepackage{a4wide}
\usepackage{enumitem}

\usepackage{amsmath,amscd}
\usepackage{amssymb}
\usepackage{amsthm}
\usepackage{comment}
\usepackage{graphicx, xcolor}
\usepackage{mathrsfs}
\usepackage{color}
\usepackage{mathrsfs}
\usepackage{multirow}
\usepackage[pagewise]{lineno}
\usepackage{subcaption}
\usepackage{cleveref}
\usepackage[ruled]{algorithm2e}
\usepackage{booktabs}
\usepackage{blkarray}
\usepackage{url} 
\numberwithin{equation}{section}

\allowdisplaybreaks

\newtheorem{theorem}{Theorem}[section]
\newtheorem{lemma}[theorem]{Lemma}
\newtheorem{definition}[theorem]{Definition}

\newtheorem{proposition}[theorem]{Proposition}

\newtheorem{remark}[theorem]{Remark}

\def\ns{\noalign{\medskip}}
\def\ds{\displaystyle}
\def\q{\quad}
\def\qq{\qquad}
\def\ms{\medskip}

\title[Exact controllability of the stochastic Maxwell equation]{Exact controllability of the stochastic Maxwell equation$:$ theory and numerical simulation}

\author[L. Sun]{Liying Sun}
\address{Academy for Multidisciplinary Studies, Capital Normal University, Beijing 100048, China.}
\email{liyingsun@lsec.cc.ac.cn}
\thanks{This work is supported by the National Natural Science Foundation of China under Grants 12471386 and 12401579.}

\author[X. Wang]{Xiaohan Wang}
\address{Academy for Multidisciplinary Studies, Capital Normal University, Beijing 100048, China.}
\email{\url{Wangxh_2025@outlook.com}}

\author[Y. Yu]{Yongyi Yu}
\address{Corresponding  author.  School of Mathematics, Jilin University, Changchun 130012, China.}
\email{yuyy122@jlu.edu.cn}
\begin{document}
\maketitle

\begin{abstract}
%This paper is devoted to studying exact controllability of three-dimensional stochastic Maxwell equations which is a coupled system consisting of two stochastic partial differential equation. First, the observability inequality for backward stochastic Maxwell equations is established via the multiplier method. Combined with duality theory, the exact controllability result of the forward equations is further proved. It is also shown that the control imposed on the diffusion term cannot be weakened. 
%{\color{blue}To verify these theoretical advancements, a numerical algorithm based on the Lagrange multiplier method is proposed, yielding corresponding results that offer practical validation and deeper insights into the underlying theoretical framework.}

%{\color{blue}This article  investigates the exact controllability of three-dimensional stochastic Maxwell equations, a coupled system comprising two stochastic partial differential equations. The research establishes the observability inequality for the backward stochastic Maxwell equations using the multiplier method, and subsequently, by leveraging duality theory, proves the exact controllability of the forward equations. A crucial finding reveals that the control applied to the diffusion term cannot be weakened, highlighting a fundamental constraint. Finally, a numerical algorithm based on the Lagrange multiplier method is proposed, yielding numerical results that offer the control value and lead to deeper insights into the underlying theoretical framework.}

This article investigates the exact controllability of three-dimensional stochastic Maxwell equations, a coupled system comprising two stochastic partial differential equations. The research establishes the observability inequality for the backward stochastic Maxwell equations using the multiplier method, and subsequently, proves the exact controllability of the forward equations. The control acting on the diffusion term is found to be indispensable, since exact controllability is destroyed when this control is removed; it is further proved that the controllability result obtained in this paper is achieved with a minimal number of controls. Finally, a numerical algorithm combining a central difference for spatial discretization, a midpoint scheme for temporal discretization and Lagrange multiplier method is proposed, yielding numerical results that offer the control value and lead to deeper insights into the underlying theoretical framework.

\end{abstract}
 \medskip
		

 \medskip
 

\section{Introduction}

Maxwell %'s
equations constitute a fundamental mathematical model for describing the evolution of electromagnetic fields. In homogeneous and isotropic materials, the classical Maxwell equations
are usually formulated in a deterministic setting, where the material parameters and source terms are assumed to be known functions. Such deterministic models have been widely used to study wave propagation, electromagnetic radiation, and related problems.
The exact controllability of deterministic Maxwell equations has been studied from several perspectives. Classical results include boundary controllability obtained by the Hilbert uniqueness method \cite{la} and %by 
stabilizability arguments \cite{ka}, as well as internal controllability established through multiplier techniques \cite{xu1} and microlocal analysis \cite{Ph}. Further extensions have addressed heterogeneous media and nonsmooth domains \cite{Ni}, topological obstructions \cite{W}, and second order Maxwell systems \cite{dgl}; see also \cite{gh} and the references therein.

However, in many practical electromagnetic phenomena, random effects cannot be neglected. These effects may arise from microscopic fluctuations of the medium, random material properties, thermal noise, or stochastic sources. To account for such uncertainties, deterministic functions in Maxwell %'s 
equations are naturally replaced by stochastic processes, leading to stochastic Maxwell equations. 
This formulation offers a robust framework for modeling electromagnetic fields in random environments and plays a crucial role in fluctuational electrodynamics, statistical radiophysics, and stochastic electromagnetism.
%Compared with the deterministic case, stochastic Maxwell equations present additional mathematical difficulties due to the interaction between the hyperbolic structure of the Maxwell %'s system 
%{\color{blue} equation} and the randomness introduced by stochastic perturbations. 
%In particular, the corresponding controllability problem becomes more delicate, since one has to control not only the drift dynamics but also the stochastic effects. 
%In contrast, the corresponding controllability problem for stochastic Maxwell equations remains much less developed.
In  contrast to their well-established deterministic counterparts, the controllability problem for stochastic Maxwell equations remains significantly less developed.
Motivated by these considerations, this paper investigates the exact controllability and numerical simulation %computation 
of controls for stochastic Maxwell equations in an isotropic conductive medium.
As a foundational step for our subsequent analysis of stochastic Maxwell equations, we first introduce the requisite notation.

Let $T>0$, and $G$ be nonempty open subset of $\mathbb{R}^3$ with a $C^1$ boundary $\Gamma$.
Set $Q=G\times (0,T)$.
%Denote by $\chi_{G_0}$ the characteristic function of the set $G_0$.
Here and henceforth, $x=(x_1,x_2,x_3)\in \mathbb{R}^3$, $\ell_{x_k}=\frac{\partial \ell}{\partial x_k}$, for $k=1,2,3$,  $\chi_t=\frac{\partial\chi}{\partial t}$, $\nabla=(\frac{\partial}{\partial x_1}, \frac{\partial}{\partial x_2}, \frac{\partial}{\partial x_3})^T$, $\nabla\times=\mbox{curl}$, $\nabla\cdot=\mbox{div}$, and $\nu(\cdot)=(\nu^1(\cdot),\nu^2(\cdot),\nu^3(\cdot))$
denotes the unit outer normal vector to $\Gamma$ at point $x$. 
Let $\left(\Omega, \mathcal{F}, \{\mathcal{F}_t\}_{t\geq 0},
\mathbb{P}\right)$ be a complete filtered probability space, on
which a one-dimensional standard Brownian motion $\{W(t)\}_{t\geq
0}$ is defined, so that $\mathbb F=\{\mathcal{F}_t\}_{t\geq 0}$ is the
natural filtration generated by $W(\cdot)$, augmented by all $\mathbb{P}$-null sets in $\mathcal{F}$.
Let $\mathcal{H}$ be a Banach space, and
$C([0,T];\mathcal{H})$ be the Banach space of all
$\mathcal{H}$-valued strongly continuous abstract functions defined on $[0,T]$. 
Denote by $L^2_{\mathcal{F}_t}(\Omega;\mathcal{H})$ the space of all $\mathcal{F}_t$-measurable random variables $\xi$ such that $\mathbb{E}|\xi|^2_{\mathcal{H}}<\infty$; 
by $L^2_{\mathbb{F}}(0,T;\mathcal{H})$ the space consisting of all $\mathcal{H}$-valued $\mathbb F$-adapted processes $X(\cdot)$ such that
$\mathbb{E}(|X(\cdot)|^2_{L^2(0,T;\mathcal{H})})<\infty$; 
by $L^\infty_{\mathbb{F}}(0,T;\mathcal{H})$ the space consisting of all $\mathcal{H}$-valued $\mathbb F$-adapted essentially bounded processes;
by $C_{\mathbb{F}}([0,T];L^2(\Omega;\mathcal{H}))$ the space consisting of all $\mathcal{H}$-valued $\mathbb F$-adapted process $X(\cdot)$ such that $X(\cdot):[0,T]\to L^2_{\mathcal{F}}(\Omega;\mathcal{H})$ is continuous; 
by $L^2_{\mathbb{F}}(\Omega; C([0,T];\mathcal{H}))$ the space consisting of all $\mathcal{H}$-valued $\mathbb F$-adapted continuous processes
$X(\cdot)$ such that $\mathbb{E}(|X(\cdot)|^2_{C([0,T];\mathcal{H})})<\infty$.
All these spaces are Banach spaces with the canonical norms (see \cite[Section 2.6]{LZ1}).
Based on the above notations, we consider
%Consider 
the following three-dimensional stochastic Maxwell equation in an isotropic conductive medium:
\begin{eqnarray}\label{1a}
\left\{
\begin{array}{ll}
\mbox{d}{\bf E}-\mbox{curl}~{\bf H}\mbox{d}t
=f\mbox{d}W(t) &\mbox{ in }Q,\\\ns\ds
\mbox{d}{\bf H}+\mbox{curl}~{\bf E}\mbox{d}t
=g\mbox{d}W(t) &\mbox{ in }Q,\\\ns\ds
\mbox{div} {\bf E}=\mbox{div} {\bf H}=0 &\mbox{ in }Q,\\ \ns\ds
{\bf E}\times \nu=u &\mbox{ on } \Gamma\times(0,T),\\ \ns\ds
{\bf E}(x,0)={\bf E}_0(x),\q  {\bf H}(x,0)={\bf H}_0(x) &\mbox{ in }G,
\end{array}
\right.
\end{eqnarray}
where ${\bf E}=(E_1, E_2, E_3)$ and  ${\bf H}=(H_1, H_2, H_3)$ represent the electric and magnetic field, respectively, $f, g, u$ are the controls, and $({\bf E}_0(x), {\bf H}_0(x))$ are the initial data. 

The main purpose of this
article is to study the exact controllability of system $(\ref{1a})$ and the numerical computation of its controls. 
More precisely, for any given initial data $({\bf E}_0(x), {\bf H}_0(x))$ and terminal data $({\bf E}_T(x), {\bf H}_T(x))$, one can find a triple of control $(f, g, u)$ such that the corresponding  solution of $(\ref{1a})$ satisfies $({\bf E}(x,T),{\bf H}(x,T))=({\bf E}_T(x), {\bf H}_T(x))$ in $G$ for some time $T>0$.

This means that we assume the time evolution of the electric field $\bf{ E}$ and magnetic field $\bf{ H}$ is driven by an externally applied current density $u$ flowing tangentially on $\Gamma$, as well as by the active manipulation of random disturbances acting on the electric field and magnetic field, represented by $f$ and $g$, respectively. 

Compared with the deterministic case, 
the controllability problem for the stochastic Maxwell equations becomes more delicate, since one has to control not only the drift dynamics but also the stochastic effects. 
On the one hand, stochastic Maxwell equations present additional mathematical difficulties due to the interaction between the hyperbolic structure of the Maxwell %'s system 
equation and the randomness introduced by stochastic perturbations. On the other hand, 
the stochastic counterpart involves several new difficulties. For instance, solutions of stochastic partial differential equations are generally non-differentiable with respect to the noise variable, and the standard compactness embedding arguments used for deterministic systems are no longer directly applicable. Moreover, the classical strategy of reducing deterministic Maxwell equations to wave equations is no longer viable in the stochastic case, owing to the presence of multiple stochastic disturbance terms in the system. 
Furthermore, unlike existing numerical approaches for the controllability of PDEs (see \cite{dgl,ez,gll,hm,mt} and references therein), the problem considered here requires the simultaneous satisfaction of divergence-free conditions for both the electric and magnetic fields. This inherent complexity is further amplified in a stochastic setting, where the numerical degrees of freedom and constraints substantially increase compared with the deterministic case, leading to a heavier computational cost.

The main contributions and findings of this paper are summarized as follows.
\begin{itemize}
\item We establish the required observability inequality for the backward stochastic Maxwell equations by means of the multiplier method, together with energy estimates and a cut-off function argument. Then, based on the duality theory, we prove the exact controllability for the forward stochastic Maxwell equations.

\item We show that, when only one equation is subject to stochastic perturbation, the stochastic Maxwell system fails to be exactly controllable for any $T>0$, even if the controls are allowed to act on the whole domain $G\times(0,T)$ and on the boundary $\Gamma\times(0,T)$. We further establish the exact controllability of the stochastic Maxwell equations with a minimal number of controls. More precisely, we prove that the system is no longer exactly controllable if any one of these controls is removed.

\item We discretize the stochastic Maxwell equations in space based on the central difference and in time based on the midpoint scheme. Leveraging the initial and terminal conditions, the control problem is then transformed into a system of linear equations for the unknown control values at discrete time nodes. The resulting linear system is subsequently solved via the Lagrange multiplier method to simulate the control triple
$(u,f,g)$ that achieves the terminal target, and the numerical results are then presented to confirm  theoretical findings.
We would like to mention that the approach of constructing the numerical algorithm for simulating control functions can be applied to other stochastic PDEs, such as stochastic wave equation, stochastic Schr\"odinger equation, stochastic heat equation, etc. 
\end{itemize}
%$\bullet$ 
%$\bullet$ 
%$\bullet$ We discretize the stochastic Maxwell equations in space using a Yee grid, and in time using a midpoint scheme. Using the initial and terminal conditions, the control problem to be solved is transformed into a system of linear equations for the control values at different time nodes. This linear system is solved using the Lagrange multiplier method, yielding the control \((u,f,g)\) that achieves the terminal target, and the numerical results are presented.

The rest of this paper is organized as follows. Section 2 presents the main results  in this paper. In Section 3, we give a pointwise identity for the stochastic Maxwell operator. Section 4 is devoted to proving the  observability estimate for the backward stochastic Maxwell equation, while Section 5 gives the detailed proof  of the  controllability result. Finally, we propose a numerical algorithm to simulate the controls of  stochastic Maxwell equations in Section 6. % Finally, we summarize the paper and discuss future  topics that are worth investigating in section 7.
\section{Problem formulation}

In this section, we give some preliminary and the main results.
First, we introduce some functional spaces which are commonly used in the analysis of stochastic Maxwell equations as follows
\begin{eqnarray*}
\begin{array}{ll}
&J_1=\left\{y\in(L^2(G))^3\Big|~\mbox{div}~y=0~\mbox{in}~G\right\},\\\ns
&J_2=\left\{y\in J_1\Big|~y\cdot \nu=0~\mbox{on}~\Gamma\right\},\\\ns
&J_3=\left\{y\in C_0^\infty(G)\Big|~y\times\nu=0~\mbox{on}~ \Gamma\right\},
\end{array}
\end{eqnarray*}
and
\begin{eqnarray*}
\begin{array}{ll}
&J^1_1=\left\{y\in J_1\Big|~{\rm curl }~y\in(L^2(G))^3,  y\times\nu=0~\mbox{on}~\Gamma\right\},\\\ns
&J^1_2=\left\{y\in J_1\Big|~{\rm curl }~y\in(L^2(G))^3,  y\cdot\nu=0~\mbox{on}~\Gamma\right\}.
\end{array}
\end{eqnarray*}
Denote by $\hat{J}_1$ the dual space of $J_1^1$ with respect to the pivot space $J_1$, and $\hat{J}_2$ the dual space of $J_2^1$ with respect to the pivot space $J_2$. 
Without loss of generality, we assume that $0\in G$ and  constants $M, M_1>0$ satisfy%, and then
\begin{eqnarray}\label{3a1}
|x|\leq M, \q \forall\ x\in G \q  \mbox{and} \q    0\leq x\cdot\nu(x)\leq M_1,\q \forall\ x\in \Gamma.
\end{eqnarray}
Let
\begin{eqnarray}\label{3ab1}
T>T^*\triangleq 18(2M^2+1).
\end{eqnarray}

%First, we state the definition of the solution and the well-posedness result for stochastic Maxwell equations.
The solution of the control system $(\ref{1a}),$ which is a stochastic PDE with a nonhomogeneous boundary  condition,
%The control system $(\ref{1a})$ is a stochastic PDE with a nonhomogeneous boundary  condition. Its solution 
is understood in the sense of transposition (see [15, section 10.2]  for example). For the readers' convenience, 
below we shall present
%let us give 
the definition of the  solution to $(\ref{1a})$. To this end, we consider the following backward stochastic Maxwell  equation
\begin{eqnarray}\label{20a}
\left\{
\begin{array}{ll}
\mbox{d}y-\mbox{curl}~z\mbox{d}t
=Y\mbox{d}W(t) &\mbox{ in }Q,\\\ns\ds
\mbox{d}z+\mbox{curl}~y\mbox{d}t=Z\mbox{d}W(t) &\mbox{ in }Q,\\\ns\ds
\mbox{div} y=\mbox{div} z=0 &\mbox{ in }Q,\\\ns\ds
y\cdot\nu=0 &\mbox{ on }\Gamma\times(0,T),\\\ns\ds
z\times\nu=0 &\mbox{ on }\Gamma\times(0,T),\\\ns\ds
y(x, \tau)=y_\tau(x),\q z(x, \tau)=z_\tau(x) &\mbox{ in }G,
\end{array}
\right.
\end{eqnarray}
where $\tau\in(0, T]$, $(y_\tau,z_\tau)\in L^2_{\mathcal{F}_\tau}(\Omega;\hat{J}_2)\times L^2_{\mathcal{F}_\tau}(\Omega; \hat{J}_1)$.
\begin{definition}\label{D1}
A quadruple of stochastic processes 
$$(y, Y, z, Z)\in C_{\mathbb{F}}([0,\tau];L^2(\Omega; \hat{J}_2))\times L^2_{\mathbb{F}}(0,\tau;(L^2(G))^3)\times C_{\mathbb{F}}([0,\tau];L^2(\Omega; \hat{J}_1))\times L^2_{\mathbb{F}}(0,\tau;(L^2(G))^3)$$ 
is called a solution of the system $(\ref{20a})$ if for any $\tau\in(0,T]$, $\phi\in J_3$ and $a.e.$ $(t, \omega)\in [0, \tau]\times\Omega$, it holds that 
\begin{equation}
\label{D1W}
\begin{aligned}
&\int_G [y(x, \tau)\cdot\phi(x)-y(x,t)\cdot\phi(x)]{\rm d}x
+\int_t^\tau\int_G z\cdot{\rm curl}~\phi{\rm d}x{\rm d}s
=\int_t^\tau\int_G \phi(x)\cdot Y(x,s){\rm d}x{\rm d}W(s),\\\ns\ds
&\int_G [z(x, \tau)\cdot\phi(x)-z(x,t)\cdot\phi(x)]{\rm d}x
-\int_t^\tau\int_G y\cdot{\rm curl}~\phi{\rm d}x{\rm d}s
=\int_t^\tau \int_G\phi(x)\cdot Z(x,s){\rm d}x{\rm d}W(s).
\end{aligned}
\end{equation}
\end{definition}
%We have the following known well-posedness result for (\ref{20a}). 
\begin{lemma}\label{l1}
For any $(y_\tau,z_\tau)\in L^2_{\mathcal{F}_\tau}(\Omega;\hat{J}_2 \times \hat{J}_1)$, the system $(\ref{20a})$  admits a unique solution $(y, Y, z, Z)$. Moreover, 
\begin{eqnarray}\label{l11}
\begin{array}{ll}
&|y|_{C_{\mathbb{F}}([0,\tau];L^2(\Omega; \hat{J}_2))}
+|Y|_{L^2_{\mathbb{F}}(0,\tau;(L^2(G))^3)}
+|z|_{C_{\mathbb{F}}([0,\tau];L^2(\Omega; \hat{J}_1))}
+|Z|_{L^2_{\mathbb{F}}(0,\tau;(L^2(G))^3)}\\\ns
&\leq C(|y_\tau|_{L^2_{\mathcal{F}_\tau}(\Omega;\hat{J}_2)}
+|z_\tau|_{L^2_{\mathcal{F}_\tau}(\Omega;\hat{J}_1)}),
\end{array}
\end{eqnarray}
where %$r_1=|a_1|_{L^\infty_{\mathbb{F}}(0,T;L^\infty(G))}+|a_2|_{L^\infty_{\mathbb{F}}(0,T;L^\infty(G))}+|b_1|_{L^\infty_{\mathbb{F}}(0,T;L^\infty(G))}+|b_2|_{L^\infty_{\mathbb{F}}(0,T;L^\infty(G))}$ and 
$C$ is a constant independent of $(y, Y, z, Z)$.
\end{lemma}

Now we %can 
give the definition of the transposition solution to $(\ref{1a})$.
\begin{definition}\label{D11}
A pair of process $({\bf E},{\bf H})\in C_{\mathbb{F}}([0,T];L^2(\Omega; J_1^1))\times C_{\mathbb{F}}([0,T];L^2(\Omega; J_2^1))$ is called a transposition solution to system {\rm(\ref{1a})} if for any $\tau\in(0,T]$ and $(y,z)$, it holds that 
\begin{eqnarray}\label{D11W}
\begin{array}{ll}
&\mathbb{E}({\bf E}(\tau),z(\tau))-({\bf E}_0,z(0))-
\mathbb{E}({\bf H}(\tau),y(\tau))+({\bf H}_0,y(0))\\\ns\ds
&\ds=\mathbb{E}\int_0^\tau \int_G (f\cdot Z-g\cdot Y){\rm d}x{\rm d}t+\mathbb{E}\int_0^\tau\int_{\Gamma}y\cdot u{\rm d}\Gamma{\rm d}t.
\end{array}
\end{eqnarray}
Here, $(y, Y, z, Z)$ solves $(\ref{20a})$  with $(y_\tau,z_\tau)\in L^2_{\mathcal{F}_\tau}(\Omega;\hat{J}_2)\times L^2_{\mathcal{F}_\tau}(\Omega; \hat{J}_1)$.
\end{definition}
With the aid of Lemma \ref{l1}, by the well-posedness result for stochastic evolution  equations with an unbounded control operators in the sense of transposition solution (see \cite{LZ1}), we immediately have the  following well-posedness result for the system (\ref{1a}).
\begin{proposition}\label{p1}
For each $({\bf E}_0, {\bf H}_0)\in J_1^1\times J_2^1$, the system {\rm(\ref{1a})} admits  a unique transposition solution $({\bf E},{\bf H})$. Moreover,
\begin{eqnarray}\label{p11W}
\begin{array}{ll}
&|({\bf E},{\bf H})|_{ C_{\mathbb{F}}([0,T];L^2(\Omega; J_1^1))\times C_{\mathbb{F}}([0,T];L^2(\Omega; J_2^1))}\\\ns\ds
&\leq C(|{\bf E}_0|_{J_1^1}+|{\bf H}_0|_{J_2^1}
+|f|_{L^2_\mathbb{F}(0,T; (L^2(G))^3)}
+|g|_{L^2_\mathbb{F}(0,T; (L^2(G))^3)}
+|u|_{L^2_\mathbb{F}(0,T; (L^2(\Gamma))^3)}).
\end{array}
\end{eqnarray}
\end{proposition}

%Now, we give 
Now we turn to the definition of the exact controllability for $(\ref{1a})$.

\begin{definition}\label{D2}
The system {\rm(\ref{1a})} is called  exactly controllable at time T if for any $({\bf E}_0, {\bf H}_0)\in J_1^1\times J_2^1$ and $({\bf E}_T,{\bf H}_T)\in L_{\mathcal{F}_T}^2(\Omega; J_1^1)\times L_{\mathcal{F}_T}^{2}(\Omega; J_2^1),$ one can find  a pair of control
$$
(f,g,u)\in L^2_\mathbb{F}(0,T; (L^2(G))^3)\times L^2_\mathbb{F}(0,T; (L^2(G))^3)\times L^2_\mathbb{F}(0,T; (L^2(\Gamma))^3),
$$
such that the solution $({\bf E}, {\bf H})$ of {\rm(\ref{1a})} corresponding to the above controls satisfies that $$({\bf E}(\cdot,T),{\bf H}(\cdot,T))=({\bf E}_T,{\bf H}_T).$$
\end{definition}

There have been several controllability results for stochastic PDEs
%partial differential equations  
(see \cite{fl, LL, l1, lw, TZ, yz, yq} and references therein). In the present paper, in a manner similar to the results in \cite{LZ1, yz}, we show that the stochastic Maxwell equation fails to be exactly controllable for any $T>0$ when stochastic noise acts only on one equation, even if the controls are allowed to act on the whole domain $G\times (0,T)$ and on the boundary $\Gamma\times (0,T)$; see Theorem \ref{T1a}. This phenomenon is in sharp contrast to the well-known controllability property of deterministic Maxwell equations. Nevertheless, for the stochastic Maxwell equation $(\ref{1a})$, we obtain the following exact controllability result:

\begin{theorem}\label{T10}
Let $T>T^*$. The %system 
stochastic Maxwell equation $(\ref{1a})$ is exactly controllable at time $T>0$.
\end{theorem}

\begin{remark}
It should be noted that the time $T$ required for the state to reach the target is restricted, and $T*$ depends on the size of the spatial domain $G$. 
Obviously, the constraint on 
$T$ here is not optimal, and this points to a direction for further investigation.
%Obviously, the constraint on $T$ here is not optimal, which is also a direction for further investigation in future work.
\end{remark}

\begin{remark}
It is worth noting that three controls $f$ and $g$ $($in the diffusion terms$)$, $u$ $($on the boundary$)$ are required  in $(\ref{1a})$ to obtain exact controllability. Moreover, the controls $f$ and $g$ in the diffusion terms are active on the whole domain. In fact, these conditions  cannot be weakened. More details will be given in Theorem {\rm \ref{T41}}.
\end{remark}

\begin{remark}
In {\rm \cite{LL, yz}}, the exact controllability of the stochastic wave equation and the stochastic beam equation is investigated.
%,  where the equations are formed by coupling a stochastic ordinary differential equation and a stochastic partial differential equation. 
However,  in this paper, we study the exact controllability of the stochastic Maxwell equation which is a coupled system consisting of two stochastic partial differential equations.  Its structure is more complex and differs significantly from that of the stochastic wave equation and the stochastic beam equation.
\end{remark}

\begin{remark}
Although the boundary control $u$ cannot be dropped in $(\ref{1a})$, it  is worthwhile to investigate its weakening to local boundary control. This seems to  be true from the existing controllability results of deterministic Maxwell systems (see \cite{dgl}).
\end{remark}

%To prove Theorem $\ref{T10}$, following the standard duality argument, we only need to prove the following observability estimate.

By the standard duality argument, the controllability result of $(\ref{1a})$ can be reduced to the following observability estimate of the adjoint system $(\ref{20a})$.
%To prove Theorem $\ref{T10}$, based on the standard duality argument, it suffices to prove the following observability estimate.

\begin{theorem}\label{T101}
There exists a constant $C>0$ such that for any solution to $(\ref{20a})$, it holds that 
\begin{eqnarray}\label{T1011}
\begin{array}{ll}
\ds\mathbb{E}\int_G \big[|y(x,T)|^2+|z(x,T)|^2\big]{\rm d}x\leq C\mathbb{E}\int_0^T\int_{\Gamma}|y|^2{\rm d}\Gamma{\rm d}t
+C\mathbb{E}\int_Q\big(|Y|^2+|Z|^2\big){\rm d}x{\rm d}t,
\end{array}
\end{eqnarray}
where $\tau=T$ in $(\ref{20a})$.
\end{theorem}

%\newpage

\section{A pointwise identity}
\label{sec;3}

This section is devoted to deriving a pointwise identity for the stochastic Maxwell operator, which plays a key role in establishing the observability estimate. To this end, we first introduce an auxiliary function. 
For any parameters $\lambda \geq 1$ and $c_1 > 0$, set
\[
\ell = \alpha \lambda^2 e^{\alpha \lambda}|x|^2 - e^{c_1 \lambda \bigl(t - \frac{T}{2}\bigr)^2},
\]
where the constant $\alpha$ satisfies
\[
\alpha \geq \max\left\{9c_1^2 T^2,\; \frac{1}{4}c_1 T^2,\; c_1 T M,\; \frac{9}{2M}c_1 T,\; \frac{1}{2}c_1 T,\; 18c_1\right\}.
\]
Then we have the following pointwise identity.

% This section is devoted to giving a pointwise identity for stochastic Maxwell operator, which will play a key role in  establishing the observability estimate.  
% Here, some auxiliary functions are introduced.
% For any parameters $\lambda\geq 1$, $c_1>0$, put
% \begin{eqnarray*}\label{3d}
%       \ell=\alpha\lambda^2e^{\alpha\lambda}|x|^2-e^{c_1\lambda(t-\frac{T}{2})^2},
% \end{eqnarray*}
% where the constant $\alpha\geq \max\{9c_1^2T^2, \frac{1}{4}c_1T^2, c_1TM, \frac{9}{2M}c_1T, \frac{1}{2}c_1T, 18c_1\}$.

% We have the following pointwise identity.

\begin{lemma}\label{221L1}
Let $y$ be an $\hat{J}_2$-valued semimartingale and $z$ be an $\hat{J}_1$-valued semimartingale.
Then, for a.e. $x\in G$, and $\mathbb{P}$-a.s. $\omega\in\Omega$,
\begin{eqnarray}\label{L1p0}
\begin{array}{ll}
&(\nabla\ell\times z-\ell_t y)\cdot\big({\rm d}y-{\rm curl}~z{\rm d}t\big)
+(-\nabla\ell\times y-\ell_t z)\cdot\big({\rm d}z+{\rm curl}~y{\rm d}t\big)\\\ns\ds
&\ds=-\frac{1}{2}{\rm d}\big(\ell_t|y|^2+\ell_t|z|^2\big)+\frac{1}{2}\ell_t|{\rm d}y|^2+\frac{1}{2}\ell_t|{\rm d}z|^2+{\rm div} V{\rm d}t
+K{\rm d}t+\nabla\cdot(z\times \ell_t y){\rm d}t\\\ns
&\ds\q+\frac{1}{2}(\ell_{tt}+\Delta\ell)|y|^2{\rm d}t
+\frac{1}{2}(\ell_{tt}+\Delta\ell)|z|^2{\rm d}t
+{\rm d}(\nabla\ell\cdot z\times y)
-\nabla\ell\cdot{\rm d}z\times{\rm d} y,
\end{array}
\end{eqnarray}
where
\begin{eqnarray*}
\left\{
\begin{array}{ll}
&\ds V=y(\nabla\ell\cdot y)-z\times(\nabla \ell\times  z)-\frac{1}{2}\nabla\ell |y|^2
+\frac{1}{2}\nabla\ell |z|^2,\\\ns
&K=-(\ell_{x_1x_1}y_1^2+\ell_{x_2x_2}y_2^2+\ell_{x_3x_3}y_3^2)-2(y_1y_2\ell_{x_1x_2}+y_1y_3\ell_{x_1x_3}+y_2y_3\ell_{x_2x_3})\\\ns
&\qq-(\ell_{x_1x_1}z_1^2+\ell_{x_2x_2}z_2^2+\ell_{x_3x_3}z_3^2)
-2(z_1z_2\ell_{x_1x_2}+z_1z_3\ell_{x_1x_3}+z_2z_3\ell_{x_2x_3}),
\end{array}
\right.
\end{eqnarray*}
and $({\rm d}y)^2$, $({\rm d}z)^2$ denote the quadratic variation processes of $y$ and $z$, respectively.
\end{lemma}

\noindent{\bf Proof.} 
A straight calculation leads that
\begin{equation}\label{L1p5}
\begin{array}{ll}
&(\nabla\ell\times z-\ell_t y)\cdot\big(\mbox{d}y-{\rm curl}~z\mbox{d}t\big)
+(-\nabla\ell\times y-\ell_t z)\cdot\big(\mbox{d}z+{\rm curl}~y\mbox{d}t\big)\\\ns
&=(\nabla\ell\times z\cdot\mbox{d}y-\nabla\ell\times y\cdot\mbox{d}z)
-(\ell_t y\cdot\mbox{d}y+\ell_t z\cdot\mbox{d}z)\\\ns
&\q-(\nabla\ell\times z\cdot\nabla\times z+\nabla\ell\times y\cdot\nabla\times y)\mbox{d}t
+\ell_t(y\cdot\nabla\times z-z\cdot\nabla\times y)\mbox{d}t.
\end{array}
\end{equation}
%Now we  analyze the right-hand side of $(\ref{L1p5})$. First, by $a\cdot b\times c=b\cdot c\times a=c\cdot a \times b$, we have
%where
Now we analyze the right-hand side of $(\ref{L1p5})$. First, by $a\cdot b\times c=b\cdot c\times a=c\cdot a \times b$, we have
\begin{equation}\label{L1p6}
\begin{array}{ll}
&\nabla\ell\times z\cdot\mbox{d}y-\nabla\ell\times y\cdot\mbox{d}z
=\nabla\ell\cdot(z\times\mbox{d}y+\mbox{d}z\times y)=\mbox{d}(\nabla\ell\cdot z\times y)
-\nabla\ell\cdot\mbox{d}z\times\mbox{d} y.
\end{array}
\end{equation} 
Second, 
\begin{equation}\label{L1p7}
\begin{array}{ll}
\ds-\ell_t y\cdot\mbox{d}y-\ell_t z\cdot\mbox{d}z
=-\frac{1}{2}\mbox{d}(\ell_t|y|^2+\ell|z|^2)
+\frac{1}{2}\ell_{tt}(|y|^2+|z|^2)\mbox{d}t
+\frac{1}{2}\ell_t(|\mbox{d}y|^2+|\mbox{d}z|^2).
\end{array}
\end{equation}
Third, based on 
%$\nabla\cdot(a\times b)=b\cdot(\nabla\times a)-a\cdot(\nabla\times b)$ and 
$\nabla\times(a\times b)=(\nabla\cdot b)a-(\nabla\cdot a)b+(b\cdot\nabla)a-(a\cdot\nabla)b,$ one can obtain that
\begin{equation}\label{L1p8}
\begin{array}{ll}
&-\nabla \ell\times z\cdot\nabla\times z
=-\nabla\cdot (z\times(\nabla \ell\times z))
-z\cdot\nabla\times(\nabla\ell\times z)\\\ns
&=-\nabla\cdot (z\times(\nabla \ell\times z))
-z\cdot[(\nabla\cdot z)\nabla\ell-(\nabla\cdot\nabla\ell)z+(z\cdot\nabla)\nabla\ell-(\nabla\ell\cdot\nabla)z].
\end{array}
\end{equation}
%By a simple calculation, 
Since ${\rm div} z=\nabla\cdot z=0$ in $Q$, we know
\begin{equation}\label{L1p882}
\begin{array}{ll}
&\ds z\cdot(z\cdot\nabla)\nabla\ell
=z\cdot\Big[\Big(z_1\frac{\partial}{\partial x_1}+z_2\frac{\partial}{\partial x_2}+z_3\frac{\partial}{\partial x_3}\Big)(\ell_{x_1}, \ell_{x_2}, \ell_{x_3})\Big]
\\\ns\ds
&=z_1(z_1\ell_{x_1x_1}+z_2\ell_{x_1x_2}+z_3\ell_{x_1x_3})+z_2(z_1\ell_{x_2x_1}+z_2\ell_{x_2x_2}\\\ns
&\quad 
+z_3\ell_{x_2x_3})+z_3(z_1\ell_{x_3x_1}+z_2\ell_{x_3x_2}+z_3\ell_{x_3x_3})\\\ns
&=z_1^2\ell_{x_1x_1}+z_2^2\ell_{x_2x_2}+z_3^2\ell_{x_3x_3}+2z_1z_2\ell_{x_1x_2}+2z_1z_3\ell_{x_1x_3}+2z_2z_3\ell_{x_2x_3},
\end{array}
\end{equation}
and
\begin{equation}\label{L1p883}
\begin{array}{ll}
&\ds z\cdot((\nabla\ell\cdot\nabla)z)
=z\cdot\Big[\Big(\ell_{x_1}\frac{\partial}{\partial x_1}+\ell_{x_2}\frac{\partial}{\partial x_2}+\ell_{x_3}\frac{\partial}{\partial x_3}\Big)(z_1, z_2, z_3)\Big]
\\\ns\ds
&=z_1(\ell_{x_1}z_{1,x_1}+\ell_{x_2}z_{1,x_2}+\ell_{x_3}z_{1,x_3})
+z_2(\ell_{x_1}z_{2,x_1}+\ell_{x_2}z_{2,x_2}+\ell_{x_3}z_{2,x_3})\\\ns
&\q+z_3(\ell_{x_1}z_{3,x_1}+\ell_{x_2}z_{3,x_2}+\ell_{x_3}z_{3,x_3})\\\ns
&\ds=\frac{1}{2}(\ell_{x_1}z_1^2+\ell_{x_1}z_2^2+\ell_{x_1}z_3^2)_{x_1}-\frac{1}{2}\ell_{x_1x_1}(z_1^2+z_2^2+z_3^2)
+\frac{1}{2}(\ell_{x_2}z_1^2+\ell_{x_2}z_2^2+\ell_{x_2}z_3^2)_{x_2}\\\ns
&\ds\q-\frac{1}{2}\ell_{x_2x_2}(z_1^2+z_2^2+z_3^2)
+\frac{1}{2}(\ell_{x_3}z_1^2+\ell_{x_3}z_2^2+\ell_{x_3}z_3^2)_{x_3}-\frac{1}{2}\ell_{x_3x_3}(z_1^2+z_2^2+z_3^2)\\\ns
&\ds=\frac{1}{2}\nabla\cdot(\nabla\ell |z|^2)-\frac{1}{2}\Delta\ell|z|^2.
\end{array}
\end{equation}
%Then, %by $(\ref{L1p8})$-$(\ref{L1p883})$, one can see that
It follows from $(\ref{L1p8})$-$(\ref{L1p883})$ that
\begin{equation}\label{L1p884}
\begin{aligned}
\ds-\nabla \ell\times z\cdot\nabla\times z
&=-\nabla\cdot (z\times(\nabla \ell\times z))
+\frac{1}{2}\Delta\ell|z|^2
+\frac{1}{2}\nabla\cdot(\nabla\ell |z|^2)\\\ns
&\quad-(z_1^2\ell_{x_1x_1}+z_2^2\ell_{x_2x_2}+z_3^2\ell_{x_3x_3})
-2(z_1z_2\ell_{x_1x_2}+z_1z_3\ell_{x_1x_3}+z_2z_3\ell_{x_2x_3}).
\end{aligned}
\end{equation}
%And, 
In addition, by the divergence condition ${\rm div}y=0$ in $Q$, we %can 
get 
\begin{equation}\label{L1p81}
\begin{array}{ll}
&\ds-\nabla\times y\cdot\nabla\ell\times y
=-\frac{1}{2}\nabla\ell\cdot\nabla(|y|^2)+\nabla\ell\cdot(y\cdot\nabla)y\\\ns\ds
&\ds=-\frac{1}{2}\nabla\cdot(\nabla\ell |y|^2)+\frac{1}{2}\Delta\ell|y|^2
+\nabla\cdot(y(\nabla\ell\cdot y))\\\ns
&\q-\ell_{x_1x_1}y_1^2-\ell_{x_2x_2}y_2^2-\ell_{x_3x_3}y_3^2-2(y_1y_2\ell_{x_1x_2}+y_1y_3\ell_{x_1x_3}+y_2y_3\ell_{x_2x_3}),
\end{array}
\end{equation}
where %the following facts were used here
we have used the following formulas
\begin{eqnarray*}
\begin{array}{ll}
&\ds(\nabla\times y) \cdot(\nabla\ell\times y)\\\ns
&\ds=(y_{3,x_2}-y_{2,x_3})(\ell_{x_2} y_3-\ell_{x_3} y_2)
+(y_{1,x_3}-y_{3,x_1})(\ell_{x_3} y_1-\ell_{x_1} y_3)+(y_{2,x_1}-y_{1,x_2})(\ell_{x_1} y_2-\ell_{x_2} y_1)\\\ns
&\ds=\ell_{x_2} y_{3,x_2}y_3-\ell_{x_2} y_{2,x_3}y_3-\ell_{x_3} y_{3,x_2}y_2+\ell_{x_3} y_{2,x_3}y_2\\\ns
&\ds\q+\ell_{x_3} y_{1,x_3}y_1-\ell_{x_3} y_{3,x_1}y_1-\ell_{x_1}y_{1,x_3}y_3+\ell_{x_1} y_{3,x_1}y_3\\\ns
&\ds\q+\ell_{x_1} y_{2,x_1}y_2-\ell_{x_1} y_{1,x_2}y_2-\ell_{x_2} y_{2,x_1}y_1+\ell_{x_2} y_{1,x_2}y_1\\\ns
&\ds\q+\ell_{x_1} y_{1,x_1}y_1+\ell_{x_2} y_{2,x_2}y_2+\ell_{x_3} y_{3,x_3}y_3-\ell_{x_1} y_{1,x_1}y_1-\ell_{x_2} y_{2,x_2}y_2-\ell_{x_3} y_{3,x_3}y_3\\\ns
&\ds=\ell_{x_1}(y_{1,x_1}y_1+y_{2,x_1}y_2+y_{3,x_1}y_3)+\ell_{x_2}(y_{1,x_2}y_1+y_{2,x_2}y_2
+y_{3,x_2}y_3)\\\ns
&\ds\q+\ell_{x_3}(y_{1,x_3}y_1+y_{2,x_3}y_2+y_{3,x_3}y_3)-\ell_{x_2}y_{2,x_3}y_3-\ell_{x_3}y_{3,x_2}y_2-\ell_{x_3}y_{3,x_1}y_1\\\ns
&\ds\q-\ell_{x_1}y_{1,x_3}y_3-\ell_{x_1}y_{1,x_2}y_2-\ell_{x_2}y_{2,x_1}y_1-\ell_{x_1}y_{1,x_1}y_1-\ell_{x_2}y_{2,x_2}y_2-\ell_{x_3}y_{3,x_3}y_3\\\ns
&\ds=\frac{1}{2}\nabla\ell\cdot\nabla(|y|^2)-\nabla\ell\cdot(y\cdot\nabla)y,
\end{array}
\end{eqnarray*}
and
\begin{eqnarray*}
\begin{array}{ll}
&\ds\nabla\ell\cdot(y\cdot\nabla)y=(\ell_{x_1},\ell_{x_2},\ell_{x_3})\cdot\Big(y_1\frac{\partial}{\partial x_1}+y_2\frac{\partial}{\partial x_2}+y_3\frac{\partial}{\partial x_3}\Big)(y_1,y_2,y_3)\\\ns
&=\ell_{x_1}(y_1y_{1,x_1}+y_2y_{1,x_2}+y_3y_{1,x_3})
+\ell_{x_2}(y_1y_{2,x_1}+y_2y_{2,x_2}+y_3y_{2,x_3})\\\ns
&\q+\ell_{x_3}(y_1y_{3,x_1}+y_2y_{3,x_2}+y_3y_{3,x_3})\\\ns
&=(y_1(\ell_{x_1}y_1+\ell_{x_2}y_2+\ell_{x_3}y_3))_{x_1}+(y_2(\ell_{x_1}y_1+\ell_{x_2}y_2+\ell_{x_3}y_3))_{x_2}+(y_3(\ell_{x_1}y_1+\ell_{x_2}y_2+\ell_{x_3}y_3))_{x_3}\\\ns
&\q-\ell_{x_1x_1}y_1^2-\ell_{x_1x_2}y_1y_2-\ell_{x_1x_3}y_1y_3-\ell_{x_1}y_1(y_{1,x_1}+y_{2,x_2}+y_{3,x_3})\\\ns
&\q-\ell_{x_2x_1}y_1y_2-\ell_{x_2x_2}y_2^2-\ell_{x_2x_3}y_2y_3-\ell_{x_2}y_2(y_{1,x_1}+y_{2,x_2}+y_{3,x_3})\\\ns
&\q-\ell_{x_3x_1}y_1y_3-\ell_{x_3x_2}y_2y_3-\ell_{x_3x_3}y_3^2-\ell_{x_3}y_3(y_{1,x_1}+y_{2,x_2}+y_{3,x_3})\\\ns
&=\nabla\cdot(y(y\cdot\nabla\ell))-(\nabla\ell\cdot y)(\nabla\cdot y)
-\ell_{x_1x_1}y_1^2-\ell_{x_2x_2}y_2^2-\ell_{x_3x_3}y_3^2\\\ns
&\q-2\ell_{x_1x_2}y_1y_2-2\ell_{x_1x_3}y_1y_3-2\ell_{x_3x_2}y_2y_3.
\end{array}
\end{eqnarray*}
Finally, together
\begin{equation*}\label{L1p9}
\begin{array}{ll}
\ell_t(y\cdot\nabla\times z-z\cdot\nabla\times y)
=\nabla\cdot(z\times \ell_t y),
\end{array}
\end{equation*}
% Together 
with the above equality with $(\ref{L1p5})$-$(\ref{L1p7})$ and $(\ref{L1p884})$-$(\ref{L1p81})$ gives the desired identity $(\ref{L1p0})$.
\hfill$\Box$

\begin{remark}
A commonly used method for studying the controllability of stochastic partial differential equations is to establish Carleman-type estimates. 
%{\color{blue} For instance,  the exact controllability of the stochastic wave equation and stochastic beam equation are investigated in {\rm \cite{LL, yz}}.}
However, when an exponential weight function is introduced and a change of variables is performed for the solution of the stochastic Maxwell equations, the divergence-free condition of the equation is no longer preserved. As a result, some terms appearing in the weighted identity cannot be eliminated. For this reason, we adopt the multiplier method here to establish a pointwise identity.
\end{remark}

%\ms

\section{Proof of the observability estimate}
This section is devoted to proving the observability estimate $(\ref{T1011})$ for the backward stochastic Maxwell equation $(\ref{20a})$ by the pointwise identity and the energy estimates. 
%Before giving the proof of Theorem $\ref{T101}$, we introduce the following two lemmas to analyze the boundary term and the time term.
First, we present the following energy estimates.
\begin{lemma}\label{L2}
For any solution 
$(y, z)$ to $(\ref{20a})$, and for all $s, t$ satisfying $0\leq s\leq t\leq T$, it holds that
\begin{eqnarray}\label{L21}
\begin{array}{ll}
&\ds\mathbb{E}\int_G(|y(x,t)|^2+|z(x,t)|^2){\rm d}x\\\ns
&\ds=\mathbb{E}\int_G(|y(x,s)|^2+|z(x,s)|^2){\rm d}x
+\frac{1}{2}\mathbb{E}\int_s^t\int_G\big(|Y(x,\tau)|^2+|Z(x,\tau)|^2\big){\rm d}x{\rm d}\tau,
\end{array}
\end{eqnarray}
and
\begin{eqnarray}\label{L211}
\begin{array}{ll}
\ds \mathbb{E}\int_G(|y(x,s)|^2+|z(x,s)|^2){\rm d}x
\leq \mathbb{E}\int_G(|y(x,t)|^2+|z(x,t)|^2){\rm d}x.
\end{array}
\end{eqnarray}
\end{lemma}

\noindent{\bf Proof. } 
Let
\begin{eqnarray*}\label{L20}
\begin{array}{ll}
\ds\mathcal{E}(t)=\frac{1}{2}\mathbb{E}\int_G(|y(x,t)|^2+|z(x,t)|^2)\mbox{d}x.
\end{array}
\end{eqnarray*}
By It\^{o}’s formula, we get that 
\begin{eqnarray}\label{L22}
\begin{array}{ll}
& \ds\mathcal{E}(t)-\mathcal{E}(s)
=%\int_s^t \frac{\mbox{d}}{\mbox{d}\tau}\mathcal{E}(\tau)\mbox{d}\tau= 
\mathbb{E}\int_s^t\int_G\big[y\cdot({\rm curl}~z)
+\frac{1}{2}|Y|^2+z\cdot(-{\rm curl}~y)
+\frac{1}{2}|Z|^2\big]\mbox{d}x\mbox{d}\tau.
\end{array}
\end{eqnarray}
Since $\nu\times z=0$ on $\Gamma\times(0,T)$, one can obtain that
\begin{eqnarray}\label{L24}
\begin{array}{ll}
&\ds\mathbb{E}\int_s^t\int_G[y\cdot\nabla\times z-z\cdot\nabla\times y]\mbox{d}x\mbox{d}\tau
=\mathbb{E}\int_s^t\int_G\nabla\cdot (z\times y)\mbox{d}x\mbox{d}\tau\\\ns
&\ds=\mathbb{E}\int_s^t\int_\Gamma\nu\cdot z\times y\mbox{d}\Gamma\mbox{d}\tau
=\mathbb{E}\int_s^t\int_\Gamma y\cdot(\nu\times z)\mbox{d}\Gamma\mbox{d}\tau=0.
\end{array}
\end{eqnarray}
Combining $(\ref{L22})$ with $(\ref{L24})$, 
\begin{eqnarray*}
\begin{array}{ll}
\ds\mathcal{E}(t)-\mathcal{E}(s)
=\frac{1}{2}\mathbb{E}\int_s^t\int_G(|Y|^2+|Z|^2)\mbox{d}x\mbox{d}\tau,
\end{array}
\end{eqnarray*}
then, one can obtain that  $(\ref{L21})$ and $(\ref{L211}).$
\hfill$\Box$

\ms

By setting 
%Let 
$$
\beta(t)=\frac{1}{2}\alpha\lambda^2e^{\alpha\lambda}-\Big(2c_1^2\lambda^2\Big(t-\frac{T}{2}\Big)^2
+c_1\lambda+\frac{1}{2M}\Big|t-\frac{T}{2}\Big|c_1\lambda\Big) e^{c_1\lambda(t-\frac{T}{2})^2},
$$
and 
$$
\gamma(t)=2M^2\alpha\lambda^2e^{\alpha\lambda}
+2Mc_1\lambda\Big|t-\frac{T}{2}\Big|e^{c_1\lambda(t-\frac{T}{2})^2},
$$
and taking advantage of 
%From 
the piontwise identity $(\ref{L1p0})$, we arrive at %have 
the following  estimate.
\begin{theorem}\label{th1}
    There exist constants $C>0$ such that for all $\lambda>1$, 
    \begin{eqnarray*}\label{th01}
    (y,z)\in C_{\mathbb{F}}([0,T];L^2(\Omega; \hat{J}_2^1))\times C_{\mathbb{F}}([0,T];L^2(\Omega; \hat{J}_1^1)), 
    \end{eqnarray*}
and $ f_1,f_2,g_1,g_2\in  L^2_{\mathbb{F}}(0,T;(L^2(G))^3)$ satisfying
\begin{eqnarray*}\label{th03}
\left\{
\begin{array}{ll}
{\rm d}y-{\rm curl}~z{\rm d}t
=f_1{\rm d}t+g_1{\rm d}W(t) &\mbox{ in }Q,\\\ns\ds
{\rm d}z+{\rm curl}~y{\rm d}t=f_2{\rm d}t+g_2{\rm d}W(t) &\mbox{ in }Q,\\\ns\ds
{\rm div} y={\rm div} z=0 &\mbox{ in }Q,\\\ns\ds
y\cdot\nu=0 &\mbox{ on }\Gamma\times(0,T),\\\ns\ds
z\times\nu=0 &\mbox{ on }\Gamma\times(0,T),
\end{array}
\right.
\end{eqnarray*}
and 
 \begin{eqnarray}\label{th04}
    y(x, 0)=y(x,T)=z(x,0)=z(x, T)=0\q  \mbox{ in }\ G,\q \mathbb{P}\mbox{-}a.s.,
    \end{eqnarray}
it holds that
\begin{eqnarray*}\label{th05}
\begin{array}{ll}
&\ds\mathbb{E}\int_{Q} \beta \big(|y|^2+|z|^2\big){\rm d}x {\rm d}t\\&\ds
\leq \mathbb{E}\int_Q\gamma\big(|f_1|^2+|f_2|^2\big){\rm d}x{\rm d}t
+C\mathbb{E}\int_Q\alpha\lambda^2e^{\alpha\lambda^2}\big(|g_1|^2+|g_2|^2\big){\rm d}x{\rm d}t\\\ns&\ds
\q+C\mathbb{E}\int_0^T\int_{\Gamma} \alpha\lambda^2e^{\alpha\lambda^2}|y|^2{\rm d}\Gamma{\rm d}t.
\end{array}
\end{eqnarray*}
\end{theorem}

\noindent{\bf Proof.}
Recalling the definition of $\ell$ in Section \ref{sec;3}, one can see that
\begin{eqnarray*}
\begin{array}{ll}
&\ds\ell_{t}=-2c_1\lambda\Big(t-\frac{T}{2}\Big)e^{c_1\lambda(t-\frac{T}{2})^2}, \ \
\ell_{tt}=-\Big[4c_1^2\lambda^2\Big(t-\frac{T}{2}\Big)^2+2c_1\lambda\Big]e^{c_1\lambda(t-\frac{T}{2})^2},\ \
\nabla\ell=2\alpha\lambda^2e^{\alpha\lambda} x,\q  \\\ns&\ds
\Delta\ell=6\alpha\lambda^2e^{\alpha\lambda},\q
\ell_{x_1x_1}=\ell_{x_2x_2}=\ell_{x_3x_3}=2\alpha\lambda^2e^{\alpha\lambda},\q
\ell_{x_1x_2}=\ell_{x_1x_3}=\ell_{x_2x_3}=0.
\end{array}
\end{eqnarray*}
From (\ref{th04}), by integrating (\ref{L1p0}) in $Q$ and taking expectation,  we have
\begin{eqnarray}\label{pT11}
\begin{array}{ll}
&\ds\mathbb{E}\int_Q\big[(\nabla\ell\times z-\ell_t y)\cdot\big(\mbox{d}y-{\rm curl}~z\mbox{d}t\big)
+(-\nabla\ell\times y-\ell_t z)\cdot\big(\mbox{d}z+{\rm curl}~y\mbox{d}t\big)\big]\mbox{d}x\\\ns\ds
&\ds=\mathbb{E}\int_Q{\rm div}V\mbox{d}x\mbox{d}t
+\mathbb{E}\int_Q [K+\nabla\cdot (z\times \ell_t y)]\mbox{d}x\mbox{d}t
+\frac{1}{2}\mathbb{E}\int_Q(\ell_{tt}
+\Delta\ell)(|y|^2+|z|^2)\mbox{d}x\mbox{d}t\\\ns\ds
&\q\ds+\mathbb{E}\int_Q\Big[\frac{1}{2}\ell_t\big(\big|\mbox{d}y\big|^2
+\big|\mbox{d}z\big|^2\big)-\nabla\ell\cdot \mbox{d}z\times \mbox{d}y\Big]\mbox{d}x.
\end{array}
\end{eqnarray}
Now, we evaluate the right-hand side of equality $(\ref{pT11})$ term by term. %For the first one, f
From $y\cdot \nu=0$ and $z\times\nu=0$ on $\Gamma\times(0,T)$, it follows that
\begin{eqnarray}\label{L13}
\begin{array}{ll}
&\ds\mathbb{E}\int_Q\nabla\cdot[y(\nabla\ell\cdot y)-z\times(\nabla\ell\times  z)]\mbox{d}x\mbox{d}t\\\ns
&\ds=\mathbb{E}\int_0^T\int_\Gamma [\nu\cdot y(\nabla\ell\cdot y)
-\nu\cdot z\times(\nabla\ell\times z)]\mbox{d}\Gamma\mbox{d}t\\\ns
&\ds=\mathbb{E}\int_0^T\int_\Gamma [(\nu\cdot y)(\nabla\ell\cdot y)
-(\nabla\ell\times z)\cdot(\nu\times z)]\mbox{d}\Gamma\mbox{d}t=0,
\end{array}
\end{eqnarray}
and
\begin{eqnarray}\label{L14}
\begin{array}{ll}
&\ds\mathbb{E}\int_Q\nabla\cdot\Big(-\frac{1}{2}\nabla\ell |y|^2
+\frac{1}{2}\nabla\ell |z|^2\Big)\mbox{d}x\mbox{d}t
=\mathbb{E}\int_0^T\int_\Gamma \alpha\lambda^2e^{\alpha\lambda^2}(\nu\cdot x) (-|y|^2+|z|^2)\mbox{d}\Gamma\mbox{d}t\\\ns
&\ds\geq -C\mathbb{E}\int_0^T\int_{\Gamma} \alpha\lambda^2e^{\alpha\lambda} |y|^2\mbox{d}\Gamma\mbox{d}t.
\end{array}
\end{eqnarray}
Regarding  the second term on the right hand side of \eqref{pT11},
\begin{eqnarray}\label{L151}
\begin{array}{ll}
&\ds\mathbb{E}\int_Q [K+\nabla\cdot (z\times \ell_t y)]\mbox{d}x\mbox{d}t\\\ns
&\ds=-\mathbb{E}\int_Q 2\alpha\lambda^2e^{\alpha\lambda}(|y|^2+|z|^2)\mbox{d}x\mbox{d}t+\mathbb{E}\int_0^T\int_\Gamma \nu\cdot z\times \ell_t y\mbox{d}x\mbox{d}t\\\ns
&\ds=-\mathbb{E}\int_Q 2\alpha\lambda^2e^{\alpha\lambda}(|y|^2+|z|^2)\mbox{d}x\mbox{d}t.
\end{array}
\end{eqnarray}
Turning to the third term,
\begin{eqnarray}\label{L152}
\begin{array}{ll}
&\ds\frac{1}{2}\mathbb{E}\int_Q(\ell_{tt}
+\Delta\ell)(|y|^2+|z|^2)\mbox{d}x\mbox{d}t\\\ns
&\ds=\mathbb{E}\int_Q \Big[3\alpha\lambda^2e^{\alpha\lambda}-\Big(2c_1^2\lambda^2\Big(t-\frac{T}{2}\Big)^2
+c_1\lambda\Big)e^{c_1\lambda(t-\frac{T}{2})^2}\Big](|y|^2+|z|^2)\mbox{d}x\mbox{d}t.
%&\ds\geq C\mathbb{E}\int_Q \alpha\lambda^2e^{\alpha\lambda^2}(|y|^2+|z|^2)\mbox{d}x\mbox{d}t.
\end{array}
\end{eqnarray}
For the fourth term, 
\begin{eqnarray}\label{L16}
\begin{array}{ll}
&\ds\frac{1}{2}\mathbb{E}\int_Q\ell_t\big(\big|\mbox{d}y\big|^2
+\big|\mbox{d}z\big|^2\big)\mbox{d}x
\!-\!\mathbb{E}\int_Q \nabla\ell\cdot \mbox{d}z\times \mbox{d}y\mbox{d}x\\\ns
&\ds=-\mathbb{E}\int_Q c_1\lambda \Big(t-\frac{T}{2}\Big)e^{c_1\lambda(t-\frac{T}{2})^2}\big(|g_1|^2+|g_2|^2\big)\mbox{d}x\mbox{d}t
\!-\!\mathbb{E}\int_Q 2\alpha\lambda^2e^{\alpha\lambda}x\cdot g_2\times g_1\mbox{d}x\mbox{d}t\\\ns
&\ds\geq-\mathbb{E}\int_Q c_1\lambda \Big(t-\frac{T}{2}\Big)e^{c_1\lambda(t-\frac{T}{2})^2}\big(|g_1|^2+|g_2|^2\big)\mbox{d}x\mbox{d}t
\!-\!\mathbb{E}\int_Q M\alpha\lambda^2e^{\alpha\lambda} \big(|g_1|^2+|g_2|^2\big) \mbox{d}x\mbox{d}t\\\ns
&\ds\geq-\mathbb{E}\int_Q (M+1)\alpha\lambda^2e^{\alpha\lambda}\big(|g_1|^2+|g_2|^2\big)\mbox{d}x\mbox{d}t.
\end{array}
\end{eqnarray}
For the left-hand side of  $(\ref{pT11})$,  since $a\cdot b\times c\leq |a||b\times c|\leq \frac{1}{2}|a|(|b|^2+|c|^2)$ and $|x|\leq M$, we %have
deduce
\begin{eqnarray}\label{L17}
\begin{array}{ll}
&\ds\mathbb{E}\int_Q\big[(\nabla\ell\times z-\ell_t y)\cdot \big(\mbox{d}y-{\rm curl}~z\mbox{d}t\big)
+(-\nabla\ell\times y-\ell_t z)\cdot\big(\mbox{d}z+{\rm curl}~y\mbox{d}t\big)\big]\mbox{d}x\\\ns
&\ds=\mathbb{E}\int_Q\big[(f_1\cdot\nabla\ell\times z-f_1\cdot\ell_t y)
+(-f_2\cdot\nabla\ell\times y-f_2\cdot\ell_t z)\big]\mbox{d}x\mbox{d}t\\\ns
&\ds\leq \frac{1}{4M}\mathbb{E}\int_Q(|\nabla\ell|+|\ell_t|)(|y|^2+|z|^2)\mbox{d}x\mbox{d}t
+M\mathbb{E}\int_Q(|\nabla\ell|+|\ell_t|)(|f_1|^2+|f_2|^2)\mbox{d}x\mbox{d}t\\\ns
&\ds\leq \mathbb{E}\int_Q\Big(\frac{1}{2}\alpha\lambda^2e^{\alpha\lambda}
+\frac{c_1\lambda}{2M}\Big|t-\frac{T}{2}\Big|e^{c_1\lambda(t-\frac{T}{2})^2}\Big)(|y|^2+|z|^2)\mbox{d}x\mbox{d}t\\\ns
&\ds\q +\mathbb{E}\int_Q\Big(2M|x|\alpha\lambda^2e^{\alpha\lambda}
+2Mc_1\lambda\Big|t-\frac{T}{2}\Big|e^{c_1\lambda(t-\frac{T}{2})^2}\Big)(|f_1|^2+|f_2|^2)\mbox{d}x\mbox{d}t.
\end{array}
\end{eqnarray}
From $(\ref{pT11})$-$(\ref{L17})$,  one %can 
obtain that
\begin{eqnarray*}\label{L18}
\begin{array}{ll}
&\ds\mathbb{E}\int_Q
\gamma(|f_1|^2+|f_2|^2)\mbox{d}x\mbox{d}t\\\ns
&\ds\geq \mathbb{E}\int_Q \Big[\frac{1}{2}\alpha\lambda^2e^{\alpha\lambda}-\Big(2c_1^2\lambda^2\Big(t-\frac{T}{2}\Big)^2
+c_1\lambda +\frac{c_1\lambda }{2M}\Big|t-\frac{T}{2}\Big|\Big)e^{c_1\lambda(t-\frac{T}{2})^2}\Big](|y|^2+|z|^2)\mbox{d}x\mbox{d}t\\\ns
&\ds\q-C\mathbb{E}\int_Q \alpha\lambda^2e^{\alpha\lambda}\big(|g_1|^2+|g_2|^2\big)\mbox{d}x\mbox{d}t
-C\mathbb{E}\int_0^T\int_{\Gamma} \alpha\lambda^2e^{\alpha\lambda}|y|^2\mbox{d}\Gamma\mbox{d}t,
\end{array}
\end{eqnarray*}
which implies the desired result.
\hfill$\Box$

\ms

For any $\lambda>1$, $T>0$, it can be found that
\begin{eqnarray}\label{3c1}
  \q  \beta_{\min}=\beta(0)=\beta(T)=\frac{1}{2}\alpha\lambda^2e^{\alpha\lambda}
    -\Big(\frac{c_1^2T^2\lambda^2}{2}+c_1\lambda+\frac{T}{4M}c_1\lambda\Big) e^{\frac{c_1\lambda T^2}{4}}\geq \frac{1}{3}\alpha\lambda^2e^{\alpha\lambda},
\end{eqnarray}
and  
\begin{eqnarray}\label{3c}
   \gamma_{\max}=\gamma(0)=\gamma(T)=2M^2\alpha\lambda^2e^{\alpha\lambda}
   +MTc_1\lambda e^{\frac{c_1\lambda T^2}{4}}\leq (2M^2+1)\alpha\lambda^2e^{\alpha\lambda}.
\end{eqnarray}

\ms

Now we construct a time-dependent cutoff function and derive estimates for its derivative.

\begin{lemma}
Let $0<aT<2aT<T-2aT<T-aT<T$. Then there exists a cut-off function
 $\chi(\cdot)\in C_0^\infty([0,T])$ such that
\begin{align}
\label{condition1}
& \chi(t)=0 \quad \text{for }\ t\in [0,aT)\cup (T-aT,T],\\[4pt]
\label{condition2}
& \chi(t)=1 \quad \text{for }\ t\in (2aT,T-2aT),
\end{align}
and
$$
|\partial_t \chi(t)|\le \frac{C}{aT}
\quad \text{for } t\in (aT,2aT)\cup (T-2aT,T-aT),
$$
where
$$
C=\frac{e^{-4}}
{\displaystyle\int_0^1 \exp\!\left(-\frac{1}{\tau(1-\tau)}\right)\,d\tau} \approx 2.6054.
$$
In particular, $C$ is a positive constant independent of $a$ and $T$.
\end{lemma}

\noindent{\bf Proof.}
Define 
$\eta(s)=
\begin{cases}
e^{-1/s}, & s>0,\\[4pt]
0, & s\le 0,
\end{cases}$ 
and set the smooth step function
\[
\theta(s)=
\begin{cases}
0, & s\le 0,\\[4pt]
\dfrac{\displaystyle\int_0^s \eta(\tau)\eta(1-\tau)\,d\tau}
{\displaystyle\int_0^1 \eta(\tau)\eta(1-\tau)\,d\tau},
& 0<s<1,\\[14pt]
1, & s\ge 1.
\end{cases}
\]
Then $\theta\in C^\infty(\mathbb R)$, $\theta(s)=0$ for $s\le 0$, and
$\theta(s)=1$ for $s\ge 1$. Moreover, for $0<s<1$,
$$
\theta'(s)
=\frac{\eta(s)\eta(1-s)}
{\displaystyle\int_0^1 \eta(\tau)\eta(1-\tau)\,d\tau}
=\frac{\exp\!\left(-\frac{1}{s}-\frac{1}{1-s}\right)}
{\displaystyle\int_0^1 \exp\!\left(-\frac{1}{\tau}-\frac{1}{1-\tau}\right)\,d\tau}
=\frac{\exp\!\left(-\frac{1}{s(1-s)}\right)}
{\displaystyle \int_0^1 \exp\!\left(-\frac{1}{\tau(1-\tau)}\right)\,d\tau}.
$$
As a result, for $s\in(0,1)$ one has $s(1-s)\le \frac14$, which yields %hence
$\exp\!\left(-\frac{1}{s(1-s)}\right)\le e^{-4},$ 
and thus
\[
|\theta'(s)|\le\frac{e^{-4}}
{\ds\int_0^1 \exp\!\left(-\frac{1}{\tau(1-\tau)}\right)\,d\tau}
=:C, \qquad s\in \mathbb R.
\]

Next we define
\[
\chi(t)=
\begin{cases}
0, & 0\leq t\le aT,\\[4pt]
\theta\!\left(\dfrac{t-aT}{aT}\right), & aT<t<2aT,\\[8pt]
1, & 2aT\le t\le T-2aT,\\[6pt]
\theta\!\left(\dfrac{T-aT-t}{aT}\right), & T-2aT<t<T-aT,\\[8pt]
0, & T-aT\le t\leq T.
\end{cases}
\]
Because $\theta$ is smooth and constant in neighborhoods of $0$ and $1$, it follows that
$\chi\in C_0^\infty([0,T])$. 
Moreover, it can be verified that \eqref{condition1} and \eqref{condition2} hold.
Now we turn to the estimate for the derivative.
For $t\in(aT,2aT)$, 
\begin{align*}
\partial_t \chi(t)
=\theta'\!\left(\frac{t-aT}{aT}\right)\frac{1}{aT}, 
\end{align*}
which yields
$|\partial_t\chi(t)|\le\frac{C}{aT}.$
Similarly, for $t\in(T-2aT,T-aT)$, 
\begin{align*}
\partial_t \chi(t)
=
\theta'\!\left(\frac{T-aT-t}{aT}\right)\left(-\frac{1}{aT}\right),
\end{align*}
and therefore $|\partial_t\chi(t)|
\le
\frac{C}{aT}.$
\hfill$\Box$

\ms

Below we give the proof of the observability estimate $(\ref{T1011})$ for the backward stochastic Maxwell equation $(\ref{20a})$.

\ms

\noindent{\bf Proof of Theorem \ref{T101}.} 
Choose some sufficiently small constant $\epsilon$ with $0<\epsilon<\frac{1}{2}$ and denote
$$
\mathcal{J}_1\triangleq [0, \epsilon T]\cup[T-\epsilon T, T],
\q \mathcal{J}_2\triangleq [2\epsilon T, T-2\epsilon T].
$$
Let $\chi\in C_0^\infty([0,T];[0,1])$ satisfy
\begin{eqnarray*}
\chi=
\left\{\!\!\!\!\!\!
\begin{array}{ll}
& 0\q \mbox{in}~ \mathcal{J}_1,\\\ns
& 1\q \mbox{in}~ \mathcal{J}_2,
\end{array}
\right.
\q \mbox{and} \q  |\chi_t|\leq \frac{C_1}{\epsilon T}\q \mbox{in}~ \mathcal{J}\triangleq(\epsilon T,2\epsilon T)\cup(T-2\epsilon T,T-\epsilon T),
\end{eqnarray*}
where the constant $C_1$ is independent of $\epsilon$, $T$, and here can be taken to be $3$.
Set $h = \chi y$ and $r = \chi z$, where \((y, z)\) satisfies \eqref{20a}. Then \((h, r)\) fulfills \(h(x,0) = h(x,T) = r(x,0) = r(x,T) = 0\) in \(G\) and solves
\begin{eqnarray*}\label{oa}
\left\{
\begin{array}{ll}
\mbox{d}h-\mbox{curl}~r\mbox{d}t
=\tilde{f}_1\mbox{d}t+\chi Y\mbox{d}W(t) &\mbox{ in }Q,\\\ns\ds
\mbox{d}r+\mbox{curl}~h\mbox{d}t=\tilde{f}_2\mbox{d}t+\chi Z\mbox{d}W(t) &\mbox{ in }Q,\\\ns\ds
\mbox{div} h=\mbox{div} r=0 &\mbox{ in }Q,\\\ns\ds
h\cdot\nu=0 &\mbox{ on }\Gamma\times(0,T),\\\ns\ds
r\times\nu=0 &\mbox{ on }\Gamma\times(0,T),
\end{array}
\right.
\end{eqnarray*}
where $\tilde{f}_1=\chi_t y$ and $\tilde{f}_2=\chi_t z$.
From Theorem $\ref{th1}$,  we have
\begin{eqnarray}\label{po01}
\begin{array}{ll}
&\ds\mathbb{E}\int_{Q} \beta \chi^2 \big(|y|^2+|z|^2\big)\mbox{d}x \mbox{d}t\\\ns
&\ds\leq \mathbb{E}\int_Q\gamma\chi_t^2\big(|y|^2+|z|^2
\big)\mbox{d}x\mbox{d}t
+C\mathbb{E}\int_Q \alpha\lambda^2e^{\alpha\lambda^2}\chi^2\big(|Y|^2+|Z|^2\big)
\mbox{d}x\mbox{d}t\\\ns&\ds
\q+C\mathbb{E}\int_0^T\int_{\Gamma} \alpha\lambda^2e^{\alpha\lambda^2}\chi^2 |y|^2\mbox{d}\Gamma\mbox{d}t.
\end{array}
\end{eqnarray}
\iffalse
There exists $\lambda_1\geq\lambda_0$, such that for all 
 $\lambda\geq\lambda_1$, it holds that 
\begin{eqnarray}\label{po02}
\begin{array}{ll}
&\ds\mathbb{E}\int_{Q} \beta\chi^2 \big(|y|^2+|z|^2\big)\mbox{d}x \mbox{d}t\\\ns
&\ds\leq C\mathbb{E}\int_Q\chi^2\big(|Y|^2+\lambda|Z|^2\big)\mbox{d}x\mbox{d}t
+C\mathbb{E}\int_Q\chi_t^2\gamma\big(|y|^2+|z|^2\big)\Big]\mbox{d}x\mbox{d}t
+C\mathbb{E}\int_0^T\int_{\Gamma}\chi^2 |y|^2\mbox{d}\Gamma\mbox{d}t\\\ns
&\ds\leq C\mathbb{E}\int_Q\big(|Y|^2+|Z|^2\big)\mbox{d}x\mbox{d}t
+C\mathbb{E}\int_{J_1}\int_G\gamma\big(|y|^2+|z|^2\big)\mbox{d}x\mbox{d}t
+C\mathbb{E}\int_0^T\int_{\Gamma} |y|^2\mbox{d}\Gamma\mbox{d}t.
\end{array}
\end{eqnarray}
\fi
Thanks to $(\ref{3c1})$ and $(\ref{3c})$, we derive
\begin{eqnarray}\label{po03}
\begin{array}{ll}
\ds \frac{1}{3}\alpha\lambda^2e^{\alpha\lambda}\mathbb{E}\int_{\mathcal{J}_2}\int_G
\big(|y|^2+|z|^2\big)\mbox{d}x \mbox{d}t
\leq \mathbb{E}\int_{Q} \beta\chi^2\big(|y|^2+|z|^2\big)\mbox{d}x \mbox{d}t,
\end{array}
\end{eqnarray}
and 
\begin{eqnarray}\label{po04}
\begin{array}{ll}
\ds \mathbb{E}\int_Q\gamma\chi_t^2\big(|y|^2+|z|^2\big)\mbox{d}x\mbox{d}t
\leq \frac{3(2M^2+1)}{\epsilon T}\alpha\lambda^2e^{\alpha\lambda}\mathbb{E}\int_{\mathcal{J}}\int_G
\big(|y|^2+|z|^2\big)\mbox{d}x\mbox{d}t.
\end{array}
\end{eqnarray}
By $(\ref{L21})$-$(\ref{L211})$ and combining $(\ref{po01})$-$(\ref{po04})$, let $\lambda=\lambda_0\geq 1$,  one can  obtain that
\begin{eqnarray*}\label{po05}
\begin{array}{ll}
&\ds \mathbb{E}\int_{G}
\big(|y(x,T)|^2+|z(x,T)|^2\big)\mbox{d}x\\\ns
&\ds\leq \frac{1}{|\mathcal{J}_2|}\mathbb{E}\int_{\mathcal{J}_2}\int_G(|y(x,t)|^2
+|z(x,t)|^2)\mbox{d}x\mbox{d}t
+C\mathbb{E}\int_Q\big(|Y|^2+|Z|^2\big)\mbox{d}x\mbox{d}t\\\ns
&\ds\leq \frac{3}{\alpha\lambda_0^2e^{\alpha\lambda_0}|\mathcal{J}_2|}\mathbb{E}\int_Q\beta\chi^2(|y(x,t)|^2+|z(x,t)|^2)\mbox{d}x\mbox{d}t
+C\mathbb{E}\int_Q\big(|Y|^2+|Z|^2\big)\mbox{d}x\mbox{d}t\\\ns
&\ds \leq \frac{3}{|\mathcal{J}_2|}\frac{3(2M^2+1)}{\epsilon T} \mathbb{E}\int_{\mathcal{J}}\int_G\big(|y|^2+|z|^2\big)\mbox{d}x \mbox{d}t
+C\mathbb{E}\int_Q\big(|Y|^2+|Z|^2\big)\mbox{d}x\mbox{d}t
+C\mathbb{E}\int_0^T\int_{\Gamma} |y|^2\mbox{d}\Gamma\mbox{d}t\\\ns
&\ds\leq \frac{3}{|\mathcal{J}_2|}\frac{3(2M^2+1)}{\epsilon T} |\mathcal{J}| \mathbb{E}\int_{G} \big(|y(x,T)|^2+|z(x,T)|^2\big)\mbox{d}x
+C\mathbb{E}\int_Q\big(|Y|^2+|Z|^2\big)\mbox{d}x\mbox{d}t\\\ns
&\ds\q
+C\mathbb{E}\int_0^T\int_{\Gamma}|y|^2\mbox{d}\Gamma\mbox{d}t,
\end{array}
\end{eqnarray*}
where $|\mathcal{J}_2|$ denotes the measure of $\mathcal{J}_2$.
Since  $\epsilon>0$ is sufficiently small and $(\ref{3ab1})$, we have 
$$
\frac{3}{|\mathcal{J}_2|}\frac{3(2M^2+1)}{\epsilon T} |\mathcal{J}|
=\frac{3}{T-4\epsilon T}\frac{3(2M^2+1)}{\epsilon T} 2\epsilon T
=\frac{18(2M^2+1)}{(1-4\epsilon)T}<1,
$$
then one can get the desired observability estimate
\begin{eqnarray*}\label{po06}
\begin{array}{ll}
\ds \mathbb{E}\int_{G}
\big(|y(x,T)|^2+|z(x,T)|^2\big)\mbox{d}x
\leq C\mathbb{E}\int_Q\big(|Y|^2+|Z|^2\big)\mbox{d}x\mbox{d}t
+C\mathbb{E}\int_0^T\int_{\Gamma}|y|^2\mbox{d}\Gamma\mbox{d}t.
\end{array}
\end{eqnarray*}
\hfill$\Box$

\begin{remark}
    As can be seen from the above proof, the restriction on the time $T$ in relation to the size of the spatial domain mainly arises from the choice of the time cut-off function. However, we believe that the value $T^*$ given in $(\ref{3ab1})$ is not sharp.
\end{remark}

\section{Controllability}
In this section,  we  prove the controllability result for the stochastic Maxwell 
equations.
By the classical duality argument and the above equality, the exact controllability of $(\ref{1a})$ is equivalent to  an observability estimate of the adjoint system $(\ref{20a})$.
\begin{proposition}\label{P1}
The system {\rm(\ref{1a})} is exactly controllable at time $T$ if and only if there is a constant $C>0$ such that $(\ref{T1011})$ holds for all the solution $(y,z)$ of $(\ref{20a})$.
\end{proposition}

The proof of Proposition \ref{P1} is based on the duality argument, as in \cite{LZ1}, and is therefore omitted here. It follows from Proposition \ref{P1} that  the controllability result can be reduced to an observability estimate.
Now, we give the proof  of Theorem {\rm\ref{T10}}.

\ms

\noindent{\bf Proof of Theorem {\rm\ref{T10}}.} Since the system $(\ref{1a})$ is linear, we need to show that for any $({\bf E}_T,{\bf H}_T)$, there exists a pair of control $(f, g, u)$ such that the corresponding solution of $(\ref{1a})$ with zero initial data $({\bf E}_0,{\bf  H}_0)=(0,0)$ satisfies $({\bf E}(x, T),{\bf H}(x, T))=({\bf E}_T(x),{\bf H}_T(x))$.  Let us introduce a linear subspace 
$$
\mathcal{X}=\left\{(y|_{\Gamma\times (0,T)}, Y, Z)|(y, z, Y, Z)~ \mbox{solves system}~ (\ref{20a})~ \mbox{with}~ (y_T,z_T)\in \hat{J}_2 \times \hat{J}_1 \right\}
$$
of $L^2_\mathbb{F}(0,T; (L^2(\Gamma))^3)\times L^2_\mathbb{F}(0,T; (L^2(G))^3)\times L^2_\mathbb{F}(0,T; (L^2(G))^3)$
and define a linear functional on $\mathcal{X}$ as follows
$$
\mathcal{L}(y|_{\Gamma\times (0,T)}, Y, Z)=\ds\mathbb{E}\int_G \big[{\bf E}_T\cdot z_T(x)
-{\bf H}_T\cdot y_T(x)\big]\mbox{d}x. 
$$ 
By means of Proposition \ref{P1}, we see that $\mathcal{L}$ is a bounded linear functional on $\mathcal{X}$. By
the Hahn–Banach theorem, $\mathcal{L}$ can be extended to a bounded linear functional on $(L^2_\mathbb{F}(0,T; (L^2(\Gamma))^3)\times L^2_\mathbb{F}(0,T; (L^2(G))^3)\times L^2_\mathbb{F}(0,T; (L^2(G))^3)$. For simplicity, we use the same notation for this
extension. Now, the Riesz representation theorem allows us to find a pair of random fields $(f, g, u)\in L^2_\mathbb{F}(0,T; (L^2(G))^3)\times L^2_\mathbb{F}(0,T; (L^2(G))^3)\times L^2_\mathbb{F}(0,T; (L^2(\Gamma))^3)$ so that
\begin{eqnarray}\label{p4211}
\begin{array}{ll}
\ds\q
\mathbb{E}\int_G \big[{\bf E}_T\cdot z_T(x)-{\bf H}_T\cdot y_T(x)\big]\mbox{d}x 
=-\mathbb{E}\int_0^T\int_{\Gamma} y\cdot u\mbox{d}\Gamma\mbox{d}t+\mathbb{E}\int_Q\big(f\cdot Z-g\cdot Y\big)\mbox{d}x\mbox{d}t.
\end{array}
\end{eqnarray}
Combining $(\ref{1a})$ with $(\ref{20a})$, one can see that
\begin{eqnarray}\label{2b11}
\begin{array}{ll}
&\ds\mathbb{E}\int_G \big[{\bf E}(x,T)\cdot z_T(x)-{\bf E}_0(x)\cdot z(x,0)-{\bf H}(x,T)\cdot y_T(x)+{\bf H}_0(x)\cdot y(x,0)\big]\mbox{d}x \\\ns&\ds
= \mathbb{E}\int_Q\big[-{\bf E}\cdot \mbox{curl}~y+z\cdot \mbox{curl}~{\bf H}+f\cdot Z
-{\bf H}\cdot \mbox{curl}~z-y\cdot \big(-\mbox{curl}~{\bf E}\big)-g\cdot Y\big]\mbox{d}x\mbox{d}t\\\ns&\ds
=\mathbb{E}\int_Q\big[\mbox{div}({\bf E}\times y)+\mbox{div}({\bf H}\times z)+f\cdot Z-g\cdot Y\big]\mbox{d}x\mbox{d}t\\\ns&\ds
=\mathbb{E}\int_0^T\int_\Gamma\big[\nu\cdot({\bf E}\times y)+\nu\cdot({\bf H}\times z)\big]\mbox{d}\Gamma\mbox{d}t+\mathbb{E}\int_Q\big(f\cdot Z-g\cdot Y\big)\mbox{d}x\mbox{d}t\\\ns&\ds
=-\mathbb{E}\int_0^T\int_{\Gamma} y\cdot u\mbox{d}\Gamma\mbox{d}t+\mathbb{E}\int_Q\big(f\cdot Z-g\cdot Y\big)\mbox{d}x\mbox{d}t.
\end{array}
\end{eqnarray}
Together $(\ref{p4211})$ with $(\ref{2b11})$, we find
\begin{eqnarray*}
\begin{array}{ll}
\ds\mathbb{E}\int_G \big[{\bf E}(x,T)\cdot z_T(x)-{\bf H}(x,T)\cdot y_T(x)\big]\mbox{d}x
=\mathbb{E}\int_G \big[{\bf E}_T\cdot z_T(x)-{\bf H}_T\cdot y_T(x)\big]\mbox{d}x.
\end{array}
\end{eqnarray*}
Since $(y_T,z_T)$ can be chosen arbitrarily, this implies that $({\bf E}(x, T),{\bf H}(x, T))=({\bf E}_T(x),{\bf H}_T(x))$ in $G$, $\mathbb P$-$a.s.$
\hfill$\Box$

\ms

\ms

At a first glance, it  seems unreasonable, especially for that the controls in the diffusion terms of $(\ref{1a})$ are acted on  the whole domain $G$. One may naturally ask whether localized controls would suffice, or whether the boundary control could be dropped. However, the answer is negative. More precisely, we consider the following stochastic Maxwell equations 
\begin{eqnarray}\label{1aa}
\left\{
\begin{array}{ll}
\mbox{d}{\bf E}-\mbox{curl}~{\bf H}\mbox{d}t
=v_1\mbox{d}t+f\mbox{d}W(t) &\mbox{ in }Q,\\\ns\ds
\mbox{d}{\bf H}+\mbox{curl}~{\bf E}\mbox{d}t
=v_2\mbox{d}t &\mbox{ in }Q,\\\ns\ds
\mbox{div} {\bf E}=\mbox{div} {\bf H}=0 &\mbox{ in }Q,\\ \ns\ds
{\bf E}\times \nu=u &\mbox{ on } \Gamma\times(0,T),\\ \ns\ds
{\bf E}(x,0)={\bf E}_0(x),\q  {\bf H}(x,0)={\bf H}_0(x) &\mbox{ in }G,
\end{array}
\right.
\end{eqnarray}
and 
\begin{eqnarray}\label{1ab}
\left\{
\begin{array}{ll}
\mbox{d}{\bf E}-\mbox{curl}~{\bf H}\mbox{d}t
=v_1\mbox{d}t &\mbox{ in }Q,\\\ns\ds
\mbox{d}{\bf H}+\mbox{curl}~{\bf E}\mbox{d}t
=v_2\mbox{d}t+g\mbox{d}W(t)%\mbox{d}t 
&\mbox{ in }Q,\\\ns\ds
\mbox{div} {\bf E}=\mbox{div} {\bf H}=0 &\mbox{ in }Q,\\ \ns\ds
{\bf E}\times \nu=u &\mbox{ on } \Gamma\times(0,T),\\ \ns\ds
{\bf E}(x,0)={\bf E}_0(x),\q  {\bf H}(x,0)={\bf H}_0(x) &\mbox{ in }G,
\end{array}
\right.
\end{eqnarray}
where $v_1, v_2, f, g, u$ are controls.
\begin{theorem}\label{T1a}
The systems $(\ref{1aa})$ and $(\ref{1ab})$ are not exactly controllable at any time $T>0$.
 \end{theorem}

In fact, both %of these 
equations $(\ref{1aa})$ and $(\ref{1ab})$ can be transformed into classical stochastic wave equation
\begin{equation*}
\begin{array}{ll}
\mbox{d}{\bf E}_t-\Delta{\bf E}\mbox{d}t
=\tilde{f}\mbox{d}t+\tilde{g}\mbox{d}W(t)%\mbox{d}t 
&\mbox{ in }Q.
\end{array}
\end{equation*}
The proof of Theorem \ref{T1a} is similar to the proof  in \cite[Theorem 10.2]{LZ1}
%of Theorem 10.2 in \cite{LZ1}, 
and thus we omit it here. It is shown that the %classical 
stochastic Maxwell equations
%system  
which only consider noise on one equation is not exactly controllable at any time, even if the controls acted everywhere on the domain $G$ and  boundary $\Gamma$.

Although it is necessary to put controls $f$ and $g$ on the whole domain, one may suspect that Theorem \ref{T10} is trivial.
However, we have the following negative result.

\begin{theorem}\label{T41}
The system $(\ref{1a})$ is not exactly controllable at any time $T>0$ provided
that one of the following three conditions is satisfied:\\[2mm]
{\rm 1)} $supp~f\subseteq G_0$, where $G_0\subsetneqq G;$\\[2mm]
{\rm 2)} $supp~g\subseteq G_0$, where $G_0\subsetneqq G;$\\[2mm]
{\rm 3)} $u=0$.
 \end{theorem}

Before proving Theorems \ref{T41}, we recall the following known result (\cite{P}).
 \begin{lemma}\label{L41}
There is a random variable $\xi\in L^2_{F_T}(\Omega)$ such that it is impossible to find $(\xi_1,\xi_2)\in$ and a constant $a\in R$ satisfying
$$
\xi=a+\int_0^T\xi_1(t){\rm d}t+\int_0^T\xi_2(t){\rm d}W(t).
$$
 \end{lemma}

\ms 

\noindent{\bf Proof of Theorem {\rm\ref{T41}}.} Let us employ the contradiction argument, and divide the proof into three cases.

\ms

 {\bf Case 1)} {\bf $f$ is support in $G_0$}. 

Since $G_0\subseteq G$ is an open subset and $G\setminus \overline{G_0}$, we can find a function $\rho\in C_0^\infty(G\setminus \overline{G_0})$ satisfying $|\rho|_{L^2(G)}=1$. Assume that $(\ref{1a})$ was exactly controllable. Then, for $({\bf E,H})=(0,0)$, one could find controls $(f,g,u)$ with $supp~f\subseteq G_0$, a.e. $(t,\omega)\in(0,T)\times\Omega$ such that the corresponding solution to $(\ref{1a})$ fulfills $({\bf E}(T),{\bf H}(T))=(\rho\xi,0)$, where $\xi$ is given in Lemma \ref{L41}. Thus, 
\begin{eqnarray*}\label{C11}
\begin{array}{ll}
\ds\rho\xi=\int_0^T<{\rm curl}~{\bf H}+f,\rho>_{H^{-1}(G),H^1(G)}\mbox{d}t+\int_0^T<g,\rho>_{L^2(G)}\mbox{d}t.
\end{array}
\end{eqnarray*}

Since $supp f\subseteq G_0$ and $\rho\in C_0^\infty(G\setminus \overline{G_0})$, then 
\begin{eqnarray*}\label{C12}
\begin{array}{ll}
\ds\rho\xi=\int_0^T<{\rm curl}~{\bf H}+g,\rho>_{H^{-1}(G),H^1(G)}\mbox{d}t,
\end{array}
\end{eqnarray*}
which contradicts Lemma \ref{L41}.

\ms

{\bf Case 2)} {\bf $f$ is support in $G_0$}. 

It is similar to the the discussion in the first case, we omit it here.

\ms

{\bf Case 3)} $u=0.$

Assume that the system $(\ref{1a})$ was exactly controllable. Similar to the observability estimate $(\ref{T1011})$, we can deduce that, for any $(y_T, z_T)$, the solution to $(\ref{20a})$ satisfies
\begin{eqnarray}\label{C13}
\begin{array}{ll}
\ds\mathbb{E}\int_G \big[|y(x,T)|^2+|z(x,T)|^2\big]{\rm d}x\leq C\mathbb{E}\int_Q\big(|Y|^2+|Z|^2\big){\rm d}x{\rm d}t,
\end{array}
\end{eqnarray}
For any nonzero data $(y_T,z_T)$, consider the following random Maxwell equation
\begin{eqnarray*}\label{C14}
\left\{
\begin{array}{ll}
\mbox{d}y-\mbox{curl}~z\mbox{d}t
=0 &\mbox{ in }Q,\\\ns\ds
\mbox{d}z+\mbox{curl}~y\mbox{d}t=0 &\mbox{ in }Q,\\\ns\ds
\mbox{div} y=\mbox{div} z=0 &\mbox{ in }Q,\\\ns\ds
y\cdot\nu=0 &\mbox{ on }\Gamma\times(0,T),\\\ns\ds
z\times\nu=0 &\mbox{ on }\Gamma\times(0,T),\\\ns\ds
y(x, T)=y_T(x),\q z(x, T)=z_T(x) &\mbox{ in }G.
\end{array}
\right.
\end{eqnarray*}
Clearly, $(y,Y,z,Z)$=$(y,0,z,0)$ solves the above system with the final datum  $(y_T,z_T)$, a contradiction to the inequality $(\ref{C13})$.
\hfill$\Box$

\ms

From the above negative result, we find that none of the two internal controls $f, g$ and boundary controls can be ignored, and  internal controls must be effective everywhere in the domain $G$.

\section{A Numerical Approach to the Exact Controllability of \eqref{1a}}

In this section, we  discuss the numerical implementation for the 
exact controllability of the stochastic Maxwell equation  $(\ref{1a})$. 
More precisely, consider \({\bf E}=(E_1,E_2,E_3)\), \({\bf H}=(H_1,H_2,H_3)\), and assume the problem is defined on the domain \(G\times (0,T)\), where \(G = (a_1,b_1) \times (a_2,b_2) \times (a_3,b_3)\), \(a_i,b_i\in \mathbb{R}\), \(i=1,2,3\). The unit outer normal vector \( \nu \) and the boundary control \(u\) are as follows:
	\[
	\begin{aligned}
		& \nu |_{x_1=a_1} = (-1,0,0),
		&& \nu |_{x_2=a_2} = (0,-1,0), 
		&& \nu |_{x_3=a_3} = (0,0,-1),\\
		& \nu |_{x_1=b_1} = (1,0,0),
		&& \nu |_{x_2=b_2} = (0,1,0),
		&& \nu |_{x_3=b_3} = (0,0,1),\\
		& u|_{x_1=a_1}= (0,-E_3,E_2) ,&& u|_{x_2=a_2}= (E_3,0,-E_1) ,&& u|_{x_3=a_3}= (-E_2,E_1,0) ,\\
		& u|_{x_1=b_1}=(0,E_3,-E_2) ,&& u|_{x_2=b_2}= (-E_3,0,E_1),&& u|_{x_3=b_3}= (E_2,-E_1,0).
	\end{aligned}
	\]
    
Our goal is  to find  %to seek 
a control \((f, g, u)\) with \(f = (f_1, f_2, f_3)\) and \(g = (g_1, g_2, g_3)\) numerically  such that for any given initial data \((\mathbf E_0(x), \mathbf H_0(x))\) and terminal data \((\mathbf E_T(x), \mathbf H_T(x))\), it holds that 
$(\mathbf E(x,T), \mathbf H(x,T)) = (\mathbf E_T(x), \mathbf H_T(x))$  in $ G $ 
for some \(T > 0\).  
%The above objectives will then be achieved by combining a central difference for spatial discretization,  a midpoint scheme for temporal discretization and Lagrange multiplier method.

% Our goal: For any given initial data \((\mathbf E_0(x),\mathbf H_0(x))\) and terminal data \((\mathbf E_T(x),\mathbf H_T(x))\) for the stochastic Maxwell equation, find the control \((f,g,u)\), with \(f = (f_1 ,f_2 ,f_3) ,g = (g_1,g_2,g_3)\), such that \((\mathbf E(x,T), \mathbf H(x,T)) = (\mathbf E_T(x), \mathbf H_T(x))\) in \(G\) for some time \(T>0\). 
% Then the above objectives will be achieved through a combination of central difference for the space descretization, and midpoint scheme  for the time descretization.

	\subsection{Algorithm Statement}
	% \cite{10.1093/imamat/18.1.9}
Based on the research of \cite{hm}, we employ the central difference on the Yee grid for spatial discretization and the midpoint scheme for temporal discretization. Through iteration, we obtain the relationship between  the terminal data \((\mathbf E_T(x), \mathbf H_T(x))\)  and both the initial data \((\mathbf E_0(x), \mathbf H_0(x))\) and the control variables to be solved. Finally, by means of %using 
the Lagrange multiplier method and the condition that the divergences of \(\bf E\), \(\bf H\) are zero, we present the numerical value of   controls for the problem.
\subsubsection{Numerical  Discretization}
Let the discrete grid of \(G\times (0,T)\) be
	\[\begin{aligned}
		&a_1=x_{1_0}<x_{1_1}<\cdots<x_{1_{N_1}}=b_1,&&a_2=x_{2_0}<x_{2_1}<\cdots<x_{2_{N_2}}=b_2,\\
		&a_3=x_{3_0}<x_{3_1}<\cdots<x_{3_{N_3}}=b_3,&&0=t_0<t_1<\cdots<t_{N_t}=T,
	\end{aligned}\]
	where \(x_{1_i}=a_1+ih_1\), \(h_1=\frac{b_1-a_1}{N_1}\), \(i=0,1,\cdots,N_1\); \(x_{2_j}=a_2+jh_2\), \(h_2=\frac{b_2-a_2}{N_2}\), \(j=0,1,\cdots,N_2\); \(x_{3_k}=a_3+kh_3\), \(h_3=\frac{b_3-a_3}{N_3}\), \(k=0,1,\cdots,N_3\); \(t_n=n\Delta t\), \(\Delta t=\frac{T}{N_t}\), \(n=0,1,\cdots,N_t.\)
Based on the Yee grid, we introduce 
	\[\begin{aligned}
		&E_1^{i+\frac{1}{2},j,k},&&i=0,\cdots,N_1-1,&&j=1,\cdots,N_2-1,&&k=1,\cdots,N_3-1,\\
		&E_2^{i,j+\frac{1}{2},k},&&i=1,\cdots,N_1-1,&&j=0,\cdots,N_2-1,&&k=1,\cdots,N_3-1,\\
		&E_3^{i,j,k+\frac{1}{2}},&&i=1,\cdots,N_1-1,&&j=1,\cdots,N_2-1,&&k=0,\cdots,N_3-1,\\
		&H_1^{i,j+\frac{1}{2},k+\frac{1}{2}},&&i=1,\cdots,N_1-1,&&j=0,\cdots,N_2-1,&&k=0,\cdots,N_3-1,\\
		&H_2^{i+\frac{1}{2},j,k+\frac{1}{2}},&&i=0,\cdots,N_1-1,&&j=1,\cdots,N_2-1,&&k=0,\cdots,N_3-1,\\
		&H_3^{i+\frac{1}{2},j+\frac{1}{2},k},&&i=0,\cdots,N_1-1,&&j=0,\cdots,N_2-1,&&k=1,\cdots,N_3-1.
	\end{aligned}\]
	% where \(E_1^{i+\frac{1}{2},j,k}\) denotes the approximation of \(E_1(x_{1_{i+\frac{1}{2}}},x_{2_j},x_{3_k})\), \(x_{1_{i+\frac{1}{2}}}=a_1+(i+\frac{1}{2})h_1\), and the quantities \(E_2^{i,j+\frac{1}{2},k}\), \(E_3^{i,j,k+\frac{1}{2}}\), \(H_1^{i,j+\frac{1}{2},k+\frac{1}{2}}\), \(H_2^{i+\frac{1}{2},j,k+\frac{1}{2}}\), \(H_3^{i+\frac{1}{2},j+\frac{1}{2},k}\) are analogous. Furthermore, the numerical approximation for \(f\) and \(g\) are defined similarly to those of \(\bf E\) and \(\bf H\), respectively. Define \(\vec{E}_1\in \mathbb{R}^{N_1(N_2-1)(N_3-1)}\) as the vector obtained by ordering the components \(E_1^{i+\frac{1}{2},j,k}\), where we first iterate over \(k\), then over \(j\), and finally over \(i\). The vectors \(\vec{E}_2\), \(\vec{E}_3\), \(\vec{H}_1\), \(\vec{H}_2\), \(\vec{H}_3\), \(\vec{f}_1\), \(\vec{f}_2\), \(\vec{f}_3\), \(\vec{g}_1\), \(\vec{g}_2\), \(\vec{g}_3\) are defined in the same way.
Let \(\vec{E}_1 \in \mathbb{R}^{N_1(N_2-1)(N_3-1)}\) be the vector whose components are the values \(E_1^{i+\frac{1}{2},j,k}\), ordered with \(k\) varying first, then \(j\), then \(i\). The vectors \(\vec{E}_2, \vec{E}_3, \vec{H}_1, \vec{H}_2, \vec{H}_3, \vec{f}_1, \vec{f}_2, \vec{f}_3, \vec{g}_1, \vec{g}_2, \vec{g}_3\) are defined analogously.
	
Considering the boundary condition, we define the numerical approximations of \(\bf E\) on the boundary as follows
\[\begin{aligned}
&E_1^{i+\frac{1}{2},0,k},&&E_1^{i+\frac{1}{2},j,0},&&E_2^{0,j+\frac{1}{2},k},&&E_2^{i,j+\frac{1}{2},0},&&E_3^{0,j,k+\frac{1}{2}},&&E_3^{i,0,k+\frac{1}{2}},\\
&E_1^{i+\frac{1}{2},N_2,k},&&E_1^{i+\frac{1}{2},j,N_3},&&E_2^{N_1,j+\frac{1}{2},k},&&E_2^{i,j+\frac{1}{2},N_3},&&E_3^{N_1,j,k+\frac{1}{2}},&&E_3^{i,N_2,k+\frac{1}{2}},
\end{aligned}\]
%	\[\begin{aligned}
%		&E_1^{i+\frac{1}{2},0,k},&&E_1^{i+\frac{1}{2},N_2,k},&&E_1^{i+\frac{1}{2},j,0},&&E_1^{i+\frac{1}{2},j,N_3},\\
%		&E_2^{0,j+\frac{1}{2},k},&&E_2^{N_1,j+\frac{1}{2},k},&&E_2^{i,j+\frac{1}{2},0},&&E_2^{i,j+\frac{1}{2},N_3},\\
%		&E_3^{0,j,k+\frac{1}{2}},&&E_3^{N_1,j,k+\frac{1}{2}},&&E_3^{i,0,k+\frac{1}{2}},&&E_3^{i,N_2,k+\frac{1}{2}},
%	\end{aligned}\]
where \(E_1^{i+\frac{1}{2},0,k}\) denotes the approximation of \(E_1(x_{1_{i+\frac{1}{2}}},0,x_{3_k})\). The remaining boundary points are similar to \(E_1^{i+\frac{1}{2},0,k}\). 
Let \(\vec{E}_{1_2}\in \mathbb{R}^{2N_1(N_3-1)}\) be the vector obtained by ordering the components \(E_1^{i+\frac{1}{2},0,k}\) and \(E_1^{i+\frac{1}{2},N_2,k}\), where we first iterate over \(E_1^{i+\frac{1}{2},0,k}\), then \(E_1^{i+\frac{1}{2},N_2,k}\), and first iterate over \(k\), then over \(j\), and finally over \(i\). 
The vectors \(\vec{E}_{1_3}\), \(\vec{E}_{2_1}\), \(\vec{E}_{2_3}\), \(\vec{E}_{3_1}\), \(\vec{E}_{3_2}\) are given in the same way.
	
Now we present the spatial discretization based on the central difference scheme, which  reads
\begin{equation}\label{eq:TEz_Maxwell_equations_matrix_form}
\left\{
\begin{aligned}
&\mbox{d}\vec{E}_1=(\tfrac{1}{h_2}A_1\vec{H}_3-\tfrac{1}{h_3}A_2\vec{H}_2)\mbox{d}t+\vec{f}_1\mbox{d}W(t),\\
&\mbox{d}\vec{E}_2=(\tfrac{1}{h_3}A_3\vec{H}_1-\tfrac{1}{h_1}A_4\vec{H}_3)\mbox{d}t+\vec{f}_2\mbox{d}W(t),\\
&\mbox{d}\vec{E}_3=(\tfrac{1}{h_1}A_5\vec{H}_2-\tfrac{1}{h_2}A_6\vec{H}_1)\mbox{d}t+\vec{f}_3\mbox{d}W(t),\\
&\mbox{d}\vec{H}_1=-(\tfrac{1}{h_2}F_1\vec{E}_3+\tfrac{1}{h_2}G_1\vec{E}_{3_2}-\tfrac{1}{h_3}F_2\vec{E}_2-\tfrac{1}{h_3}G_2\vec{E}_{2_3})\mbox{d}t+\vec{g}_1\mbox{d}W(t),\\
&\mbox{d}\vec{H}_2=-(\tfrac{1}{h_3}F_3\vec{E}_1+\tfrac{1}{h_3}G_3\vec{E}_{1_3}-\tfrac{1}{h_1}F_4\vec{E}_3-\tfrac{1}{h_1}G_4\vec{E}_{3_1})\mbox{d}t+\vec{g}_2\mbox{d}W(t),\\
&\mbox{d}\vec{H}_3=-(\tfrac{1}{h_1}F_5\vec{E}_2+\tfrac{1}{h_1}G_5\vec{E}_{2_1}-\tfrac{1}{h_2}F_6\vec{E}_1-\tfrac{1}{h_2}G_6\vec{E}_{1_2})\mbox{d}t+\vec{g}_3\mbox{d}W(t),
\end{aligned}
\right.
\end{equation}
where %the matrices in \eqref{eq:TEz_Maxwell_equations_matrix_form} are as follows. 
\[A_1=\operatorname{diag}\underbrace{(D_1,\cdots,D_1)}_{\text{\(N_1\) blocks}},\quad A_2=\operatorname{diag}\underbrace{(D_2,\cdots,D_2)}_{\text{\(N_1(N_2-1)\) blocks}},\quad A_3=\operatorname{diag}\underbrace{(D_2,\cdots,D_2)}_{\text{\((N_1-1)N_2\) blocks}},\]
	\[D_1=\begin{blockarray}{@{}ccccc@{}}
		& \scriptscriptstyle 1 & \scriptscriptstyle 2 &\cdots  & \scriptscriptstyle N_2 \\
		\begin{block}{c [c c c c]}
			\scriptscriptstyle 1      & -I_{N_3-1} & I_{N_3-1} & & \\
			\vdots& & \ddots & \ddots & \\
			\scriptscriptstyle N_2-1  & & & -I_{N_3-1} & I_{N_3-1} \\
		\end{block}
	\end{blockarray},\quad D_2=\begin{blockarray}{@{}ccccc@{}}
		& \scriptscriptstyle 1 & \scriptscriptstyle 2 &\cdots  & \scriptscriptstyle N_3 \\
		\begin{block}{c [c c c c]}
			\scriptscriptstyle 1      & -1 & 1 & & \\
			\vdots& & \ddots & \ddots & \\
			\scriptscriptstyle N_3-1  & & & -1 & 1 \\
		\end{block}
	\end{blockarray},\]
with \(I\) being the identity matrix, for example \(I_{N_3}\in \mathbb{R}^{N_3\times N_3}\), and \(0_{m\times n}\) being the \(m\times n\) zero matrix. 
In addition, \(A_4\) and \(A_5\) are block matrices as follows
	\[
	\begin{blockarray}{@{}ccccc@{}}
		& \scriptscriptstyle 1 & \scriptscriptstyle 2 & \cdots &\scriptscriptstyle N_1 \\
		\begin{block}{c [c c c c]}
			\scriptscriptstyle 1      & -I_{M} & I_{M} & & \\
			\vdots& & \ddots & \ddots & \\
			\scriptscriptstyle N_1-1  & & & -I_{M} & I_{M} \\
		\end{block}
	\end{blockarray},\]
	where \(M=N_2(N_3-1)\) for the case of \(A_4\), and \(M=(N_2-1)N_3\) for the case of \(A_5\), respectively, and 
	\[A_6=\operatorname{diag}\underbrace{(D_1',\cdots,D_1')}_{\text{\(N_1-1\) blocks}},\quad D_1'=\begin{blockarray}{@{}ccccc@{}}
		& \scriptscriptstyle 1 & \scriptscriptstyle 2 & \cdots & \scriptscriptstyle N_2 \\
		\begin{block}{c [c c c c]}
			\scriptscriptstyle 1      & -I_{N_3} & I_{N_3} & & \\
			\vdots& & \ddots & \ddots & \\
			\scriptscriptstyle N_2-1  & & & -I_{N_3} & I_{N_3} \\
		\end{block}
	\end{blockarray}.\] 
Moreover, the matrices \(F_1,\cdots,F_6\) satisfy
\[F_1=-A_6^{\top},\quad F_2=-A_3^{\top},\quad F_3=-A_2^{\top},\quad F_4 =-A_5^{\top},\quad F_5 =-A_4^{\top},\quad F_6=-A_1^{\top}.\]
The matrices \(G_1,\cdots,G_6\) are
	\[G_1=\begin{bmatrix}
		G_{11}&G_{12}
	\end{bmatrix},\quad G_2=\begin{bmatrix}
		G_{21}&G_{22}
	\end{bmatrix},\quad G_3=\begin{bmatrix}
		G_{31}&G_{32}
	\end{bmatrix},\]
	\[G_4=\begin{bmatrix}
		G_{41}&G_{42}
	\end{bmatrix},\quad G_5=\begin{bmatrix}
		G_{51}&G_{52}
	\end{bmatrix}, \quad G_6=\begin{bmatrix}
		G_{61}&G_{62}
	\end{bmatrix},\]
with
	\[\begin{aligned}
		&G_{11}=\operatorname{diag}\underbrace{(D_3,\cdots,D_3)}_{\text{\(N_1-1\) blocks}},&&G_{21}=\operatorname{diag}\underbrace{(D_5,\cdots,D_5)}_{\text{\((N_1-1)N_2\) blocks}},&&G_{31}=\operatorname{diag}\underbrace{(D_5,\cdots,D_5)}_{\text{\(N_1(N_2-1)\) blocks}},\\
		&G_{12}=\operatorname{diag}\underbrace{(D_4,\cdots,D_4)}_{\text{\(N_1-1\) blocks}},&&G_{22}=\operatorname{diag}\underbrace{(D_6,\cdots,D_6)}_{\text{\((N_1-1)N_2\) blocks}},&&G_{32}=\operatorname{diag}\underbrace{(D_6,\cdots,D_6)}_{\text{\(N_1(N_2-1)\) blocks}},\\
	\end{aligned}\]
	\[G_{41}=\begin{bmatrix}
		-I_{(N_2-1)N_3}\\0_{(N_1-1)(N_2-1)N_3\times (N_2-1)N_3}
	\end{bmatrix}, \;G_{51}=\begin{bmatrix}
		-I_{N_2(N_3-1)}\\0_{(N_1-1)N_2(N_3-1)\times N_2(N_3-1)}
	\end{bmatrix},\;G_{61}=\operatorname{diag}\underbrace{(D_3',\cdots,D_3')}_{\text{\(N_1\) blocks}},\]
	\[G_{42}=\begin{bmatrix}
		0_{(N_1-1)(N_2-1)N_3\times (N_2-1)N_3}\\I_{(N_2-1)N_3}
	\end{bmatrix},\;
	G_{52}=\begin{bmatrix}
		0_{(N_1-1)N_2(N_3-1)\times N_2(N_3-1)}\\I_{N_2(N_3-1)}
	\end{bmatrix}, \;G_{62}=\operatorname{diag}\underbrace{(D_4',\cdots,D_4')}_{\text{\(N_1\) blocks}},\]
	\[\begin{aligned}
		&D_3=\begin{bmatrix}
			-I_{N_3}\\0_{(N_2-1)N_3\times N_3}
		\end{bmatrix},&& D_5=\begin{bmatrix}
			-1\\0_{(N_3-1)\times 1}
		\end{bmatrix},&&D_3'=\begin{bmatrix}
		-I_{N_3-1}\\0_{(N_2-1)(N_3-1)\times (N_3-1)}
		\end{bmatrix},\\\ns
		&D_4=\begin{bmatrix}
		0_{(N_2-1)N_3\times N_3}\\I_{N_3}
		\end{bmatrix},&&D_6=\begin{bmatrix}
			0_{(N_3-1)\times 1}\\1
		\end{bmatrix}, &&D_4'=\begin{bmatrix}
		0_{(N_2-1)(N_3-1)\times (N_3-1)}\\I_{N_3-1}
		\end{bmatrix}.
	\end{aligned}\]
	
For simplicity, we combine the components and define	
\[\vec{E}_h(t)=\begin{bmatrix}
		(\vec{E}_1(t))^{\top}&(\vec{E}_2(t))^{\top}&(\vec{E}_3(t))^{\top}
	\end{bmatrix}^{\top},\]
	\[\vec{E}_*(t)=\begin{bmatrix}
		(\vec{E}_{1_2}(t))^{\top}&(\vec{E}_{1_3}(t))^{\top}&(\vec{E}_{2_1}(t))^{\top}&(\vec{E}_{2_3}(t))^{\top}&(\vec{E}_{3_1}(t))^{\top}&(\vec{E}_{3_2}(t))^{\top}
	\end{bmatrix}^{\top},\]
	\[A=\begin{bmatrix}
		0&-\tfrac{1}{h_3}A_2&\tfrac{1}{h_2}A_1\\\tfrac{1}{h_3}A_3&0&-\tfrac{1}{h_1}A_4\\-\tfrac{1}{h_2}A_6&\tfrac{1}{h_1}A_5&0
	\end{bmatrix},\quad F=\begin{bmatrix}
		0&\tfrac{1}{h_3}F_2&-\tfrac{1}{h_2}F_1\\-\tfrac{1}{h_3}F_3&0&\tfrac{1}{h_1}F_4\\\tfrac{1}{h_2}F_6&-\tfrac{1}{h_1}F_5&0
	\end{bmatrix},\]
	\[G=\begin{bmatrix}
		0&0&0&\tfrac{1}{h_3}G_2&0&-\tfrac{1}{h_2}G_1\\0&-\tfrac{1}{h_3}G_3&0&0&\tfrac{1}{h_1}G_4&0\\\tfrac{1}{h_2}G_6&0&-\tfrac{1}{h_1}G_5&0&0&0
	\end{bmatrix}.\]
The vectors \(\vec{H}_h(t)\), \(\vec{f}_h(t)\), \(\vec{g}_h(t)\) are defined analogously.	We denote by \(\vec{E}_{h,n}\) the approximation of \(\vec{E}_h(t_n)\); similarly, \(\vec{H}_{h,n}\), \(\vec{f}_{h,n}\), \(\vec{g}_{h,n}\), \(\vec{E}_{*,n}\), \(\vec{H}_{*,n}\)  approximate the corresponding quantities at \(t_n\). Let \(\Delta W_n = W(t_{n+1}) - W(t_n)\) be the increment of the Wiener process.	
% We use \(\vec{E}_{h,n}\) to approximate \(\vec{E}_h(t_n)\). \(\vec{H}_{h,n}\), \(\vec{f}_{h,n}\), \(\vec{g}_{h,n}\), \(\vec{E}_{*,n}\) are similar to \(\vec{E}_{h,n}\). And denote \(\Delta W_n=W(t_{n+1})-W(t_n)\).  
Then, applying the midpoint scheme to discretize \eqref{eq:TEz_Maxwell_equations_matrix_form}, we obtain the following full discretization
\begin{equation}\label{fully_discrete}
		\left\{\begin{aligned}
			\vec{E}_{h,n+1}-\tfrac{\Delta t}{2}A\vec{H}_{h,n+1}&=\vec{E}_{h,n}+\tfrac{\Delta t}{2}A\vec{H}_{h,n}+\vec{f}_{h,n}\Delta W_n,\\
			-\tfrac{\Delta t}{2}F\vec{E}_{h,n+1}+\vec{H}_{h,n+1}&=\tfrac{\Delta t}{2}F\vec{E}_{h,n}+\vec{H}_{h,n}+\vec{g}_{h,n}\Delta W_n+\tfrac{\Delta t}{2}G\vec{E}_{*,n+1}+\tfrac{\Delta t}{2}G\vec{E}_{*,n}.
		\end{aligned}\right.
	\end{equation}

Note that the divergence-free property of \(\bf E\) and \(\bf H\) should be hold. To this end, we introduce the discrete divergence operator
\[\begin{aligned}
(\nabla_h\cdot \vec{E}_h)^{i,j,k}=\tfrac{1}{h_1}(E_1^{i+\frac{1}{2},j,k}-E_1^{i-\frac{1}{2},j,k})+\tfrac{1}{h_2}(E_2^{i,j+\frac{1}{2},k}-E_2^{i,j-\frac{1}{2},k})+\tfrac{1}{h_3}(E_3^{i,j,k+\frac{1}{2}}-E_3^{i,j,k-\frac{1}{2}}),
\end{aligned}\]
\[\begin{aligned}
(\nabla_h\cdot{\vec{H}_h})^{i+\frac{1}{2},j+\frac{1}{2},k+\frac{1}{2}}&=\tfrac{1}{h_1}(H_1^{i+1,j+\frac{1}{2},k+\frac{1}{2}}-H_1^{i,j+\frac{1}{2},k+\frac{1}{2}})
+\tfrac{1}{h_2}(H_2^{i+\frac{1}{2},j+1,k+\frac{1}{2}}-H_2^{i+\frac{1}{2},j,k+\frac{1}{2}})\\
&+\tfrac{1}{h_3}(H_3^{i+\frac{1}{2},j+\frac{1}{2},k+1}-H_3^{i+\frac{1}{2},j+\frac{1}{2},k}).
\end{aligned}\]
Due to the discrete  version $\nabla_h\cdot\vec{E}_{h,n}=0$ and $\nabla_h\cdot\vec{H}_{h,n}=0,$
%yields that the divergence-free requirement for \(\mathbf{E}\) and \(\mathbf{H}\) leads to satisfy the divergence-free condition for \(\mathbf{E}\) and \(\mathbf{H}\),
we require that \(\vec{f}_{h,n}\) and \(\vec{g}_{h,n}\) satisfy the following equality  for \(n = 0, \dots, N_t-1\):
% To satisfy the divergence-free condition for \(\bf E\) and \(\bf H\), we require \(f\) and \(g\) to satisfy the following equality for \(t_n\), \(n=0,\cdots, N_t-1\). 
\begin{equation}\label{div_f_g}
\nabla_h\cdot \vec{f}_{h,n}=0,\quad \nabla_h \cdot \vec{g}_{h,n}=0.
\end{equation}
Define the numerical solutions of \(g\) on the boundary by
\begin{align*}
&g_{1,n}^{0,j+\frac{1}{2},k+\frac{1}{2}},\quad g_{1,n}^{N_1,j+\frac{1}{2},k+\frac{1}{2}},\quad
g_{2,n}^{i+\frac{1}{2},0,k+\frac{1}{2}},\\
&g_{2,n}^{i+\frac{1}{2},N_2,k+\frac{1}{2}},\quad 
g_{3,n}^{i+\frac{1}{2},j+\frac{1}{2},0},\quad g_{3,n}^{i+\frac{1}{2},j+\frac{1}{2},N_3},
\end{align*}
where \(g_{1,n}^{0,j+\frac{1}{2},k+\frac{1}{2}}\) denotes the approximation of \(g_1(t_n)\) at the node $(0,x_{2_{j+\frac{1}{2}}},x_{3_{k+\frac{1}{2}}})$. 
The remaining boundary points are similar to \(g_{1,n}^{0,j+\frac{1}{2},k+\frac{1}{2}}\). Define \(\vec{g}_{1_1,n}\in \mathbb{R}^{2N_2N_3}\) as the vector obtained by ordering the components \(g_{1,n}^{0,j+\frac{1}{2},k+\frac{1}{2}}\) and \(g_{1,n}^{N_1,j+\frac{1}{2},k+\frac{1}{2}}\), where we iterate over \(g_{1,n}^{0,j+\frac{1}{2},k+\frac{1}{2}}\) first, then \(g_{1,n}^{N_1,j+\frac{1}{2},k+\frac{1}{2}}\), and iterate first over \(k\), then over \(j\), and finally over \(i\). The vectors \(\vec{g}_{2_2,n}\) and \(\vec{g}_{3_3,n}\) are defined in the same way.
Then, \eqref{div_f_g} can be rewritten as
\begin{equation}\label{matrix_div}
\left\{\begin{aligned}
&\tfrac{1}{h_1}V_1\vec{f}_{1,n}+\tfrac{1}{h_2}V_2\vec{f}_{2,n}+\tfrac{1}{h_3}V_3\vec{f}_{3,n}=0,\\
&\tfrac{1}{h_1}P_1\vec{g}_{1,n}+\tfrac{1}{h_1}Q_1\vec{g}_{1_1,n}+\tfrac{1}{h_2}P_2\vec{g}_{2,n}+\tfrac{1}{h_2}Q_2\vec{g}_{2_2,n}+\tfrac{1}{h_3}P_3\vec{g}_{3,n}+\tfrac{1}{h_3}Q_3\vec{g}_{3_3,n}=0,
\end{aligned}
\right.
\end{equation}
where \(V_2=\operatorname{diag}\underbrace{(D_1,\cdots,D_1)}_{\text{\((N_1-1)\) blocks}},\quad V_3=\operatorname{diag}\underbrace{(D_2,\cdots,D_2)}_{\text{\((N_1-1)(N_2-1)\) blocks}},\) and \[V_1=
	\begin{blockarray}{@{}ccccc@{}}
		& \scriptscriptstyle 1 & \scriptscriptstyle 2 & \cdots &\scriptscriptstyle N_1 \\
		\begin{block}{c [c c c c]}
			\scriptscriptstyle 1      & -I_{M} & I_{M} & & \\
			\vdots& & \ddots & \ddots & \\
			\scriptscriptstyle N_1-1  & & & -I_{M} & I_{M} \\
		\end{block}
	\end{blockarray},\quad M=(N_2-1)(N_3-1).\] 
In addition, the matrices \(P_1,P_2,P_3\) are
	\[P_1 =
	\begin{blockarray}{@{}cccc@{}}
		& \scriptscriptstyle 1 & \cdots  &\scriptscriptstyle N_1-1 \\
		\begin{block}{c [c c c]}
			\scriptscriptstyle 1      & I_{N_2N_3} &  & \\
			\scriptscriptstyle 2 &-I_{N_2N_3}&\ddots&\\
			\vdots& & \ddots & I_{N_2N_3} \\
			\scriptscriptstyle N_1  & & & -I_{N_2N_3} \\
		\end{block}
	\end{blockarray},\,\quad
    \begin{aligned}
& P_2=\operatorname{diag}\underbrace{(-(D_1')^{\top},\cdots,-(D_1')^{\top})}_{\text{\(N_1\) blocks}},\\\ns
& P_3=\operatorname{diag}\underbrace{(-D_2^{\top},\cdots,-D_2^{\top})}_{\text{\(N_1N_2\) blocks}}.
    \end{aligned}
    % P_2=\operatorname{diag}\underbrace{(-(D_1')^{\top},\cdots,-(D_1')^{\top})}_{\text{\(N_1\) blocks}},\,P_3=\operatorname{diag}\underbrace{(-D_2^{\top},\cdots,-D_2^{\top})}_{\text{\(N_1N_2\) blocks}}.
    \]
Furthermore, the matrices \(Q_1,Q_2,Q_3\) are
	\[Q_1=\begin{bmatrix}
		Q_{11}&Q_{12}
	\end{bmatrix},\quad Q_2=\begin{bmatrix}
		Q_{21}&Q_{22}
	\end{bmatrix},\quad Q_3=\begin{bmatrix}
		Q_{31}&Q_{32}
	\end{bmatrix},\]
	where 
	\begin{align*}
&Q_{11}=\begin{bmatrix}
			-I_{N_2N_3}\\0_{(N_1-1)N_2N_3\times N_2N_3}
		\end{bmatrix},\quad \q~
Q_{12}=\begin{bmatrix}
			0_{(N_1-1)N_2N_3\times N_2N_3}\\I_{N_2N_3}
		\end{bmatrix},\\\ns
&Q_{21}=\operatorname{diag}\underbrace{(D_3,\cdots,D_3)}_{\text{\(N_1\) blocks}},\quad \q\q Q_{22}=\operatorname{diag}\underbrace{(D_4,\cdots,D_4)}_{\text{\(N_1\) blocks}},\\
&Q_{31}=\operatorname{diag}\underbrace{(D_5,\cdots,D_5)}_{\text{\(N_1N_2\) blocks}},\quad \q\q Q_{32}=\operatorname{diag}\underbrace{(D_6,\cdots,D_6)}_{\text{\(N_1N_2\) blocks}}.
\end{align*}
	% \[\begin{aligned}
	% 	&Q_{11}=\begin{bmatrix}
	% 		-I_{N_2N_3}\\0_{(N_1-1)N_2N_3\times N_2N_3}
	% 	\end{bmatrix},&&Q_{21}=\operatorname{diag}\underbrace{(D_3,\cdots,D_3)}_{\text{\(N_1\) blocks}},&&Q_{31}=\operatorname{diag}\underbrace{(D_5,\cdots,D_5)}_{\text{\(N_1N_2\) blocks}},\\
	% 	&Q_{12}=\begin{bmatrix}
	% 		0_{(N_1-1)N_2N_3\times N_2N_3}\\I_{N_2N_3}
	% 	\end{bmatrix},&&Q_{22}=\operatorname{diag}\underbrace{(D_4,\cdots,D_4)}_{\text{\(N_1\) blocks}},&&Q_{32}=\operatorname{diag}\underbrace{(D_6,\cdots,D_6)}_{\text{\(N_1N_2\) blocks}}.
	% \end{aligned}\]

Next, we assemble the components across all time instants. Define
 \[\vec{g}_{*,n}=\begin{bmatrix}
  (\vec{g}_{1_1,n})^{\top}&(\vec{g}_{2_2,n})^{\top}&(\vec{g}_{3_3,n})^{\top}
 \end{bmatrix}^{\top},\quad \vec{f}_{h,s}=\begin{bmatrix}
 (\vec{f}_{h,0})^{\top}&\cdots&(\vec{f}_{h,N_t-1})^{\top}
 \end{bmatrix}^{\top},\]
 \[V_0=\begin{bmatrix}
  V_1&V_2&V_3
 \end{bmatrix},\quad P_0=\begin{bmatrix}
  P_1&P_2&P_3
 \end{bmatrix},\quad Q_0=\begin{bmatrix}
  Q_1&Q_2&Q_3
 \end{bmatrix},\]
 and let \(\vec{g}_{h,s}\), \(\vec{g}_{*,s}\) be defined analogously to 
 \(\vec{f}_{h,s}\).
 Then, we have the following discrete divergence conditions
\begin{equation}\label{div_condition}
  \left\{\begin{aligned}
   &V\vec{f}_{h,s}=0,\\
   &P\vec{g}_{h,s}+Q\vec{g}_{*,s}=0,
  \end{aligned}\right.
 \end{equation}
 where \(V=\operatorname{diag}\underbrace{(V_0,\cdots,V_0)}_{\text{\(N_t\) blocks}}\), \(P=\operatorname{diag}\underbrace{(P_0,\cdots,P_0)}_{\text{\(N_t\) blocks}}\), \(Q=\operatorname{diag}\underbrace{(Q_0,\cdots,Q_0)}_{\text{\(N_t\) blocks}}\). 
We now incorporate all boundary points into \eqref{fully_discrete}. Let \(\vec{R}_{*,n}\) approximate \(\vec{R}_*(t_n)\), where \(\vec{R}_*(t_n)\) is defined as the column vector of boundary values of \(\mathbf{E}\) and \(\mathbf{H}\) excluding those already contained in \(\vec{E}_*(t_n)\). The number of components of \(\vec{R}_{*,n}\) is denoted by \(N_R\), and  the sizes and notations of the entries in the various vectors are
\[N_E=N_1(N_2-1)(N_3-1)+(N_1-1)N_2(N_3-1)+(N_1-1)(N_2-1)N_3,\]
\[N_H=(N_1-1)N_2N_3+N_1(N_2-1)N_3+N_1N_2(N_3-1),\]
\[N_*=2N_1(N_2+N_3-2)+2N_2(N_1+N_3-2)+2N_3(N_1+N_2-2).\]
Collecting all variables that include every grid point and letting
\[\Phi_n=\begin{bmatrix}
\vec{E}_{h,n}\\\vec{H}_{h,n}\\\vec{E}_{*,n}\\\vec{R}_{*,n}
	\end{bmatrix},\quad M_1=\begin{bmatrix}
		I_{N_E}&-\tfrac{\Delta t}{2}A&0&0\\-\tfrac{\Delta t}{2}F&I_{N_H}&0&0\\0&0&I_{N_*}&0\\0&0&0&I_{N_R}
	\end{bmatrix},\quad M_2=\begin{bmatrix}
		I_{N_E}&\tfrac{\Delta t}{2}A&0&0\\\tfrac{\Delta t}{2}F&I_{N_H}&0&0\\0&0&0&0\\0&0&0&0
	\end{bmatrix},\]
	\[B_f=\begin{bmatrix}
		I_{N_E}&0&0&0
	\end{bmatrix}^{\top},\quad B_g=\begin{bmatrix}
		0&I_{N_H}&0&0
	\end{bmatrix}^{\top},\quad B_1=\begin{bmatrix}
		0& \tfrac{\Delta t}{2}G^{\top} &I_{N_*}&0
	\end{bmatrix}^{\top},\]
	\[B_2=\begin{bmatrix}
		0& \tfrac{\Delta t}{2}G^{\top} &0&0
	\end{bmatrix}^{\top},\quad B_3=\begin{bmatrix}
		0&0&0&I_{N_R}
	\end{bmatrix}^{\top},\]
we arrive at
\[
M_1\Phi_{n+1}=M_2\Phi_n + \bigl(B_f\vec{f}_{h,n}+B_g\vec{g}_{h,n}\bigr)\Delta W_n + B_1\vec{E}_{*,n+1}+B_2\vec{E}_{*,n}+B_3\vec{R}_{*,n+1}.
\]
It can be found that \(M_1\) is invertible.  Set \(K = M_1^{-1}\) and \(U = K M_2\). Then for \(n=0,\dots,N_t-1\),
\begin{equation}
\Phi_{n+1}=U\Phi_n + K\Bigl(\bigl(B_f\vec{f}_{h,n}+B_g\vec{g}_{h,n}\bigr)\Delta W_n + B_1\vec{E}_{*,n+1}+B_2\vec{E}_{*,n}+B_3\vec{R}_{*,n+1}\Bigr).
\end{equation}
Iterating this recurrence gives
\begin{equation}\label{k_is_m_1}
\begin{aligned}
\Phi_{N_t}=U^{N_t}\Phi_{0}
&+U^{N_t-1}K\Bigl(\bigl(B_f\vec{f}_{h,0}+B_g\vec{g}_{h,0}\bigr)\Delta W_{0}+B_1\vec{E}_{*,1}+B_2\vec{E}_{*,0}+B_3\vec{R}_{*,1}\Bigr)\\
&+U^{N_t-2}K\Bigl(\bigl(B_f\vec{f}_{h,1}+B_g\vec{g}_{h,1}\bigr)\Delta W_{1}+B_1\vec{E}_{*,2}+B_2\vec{E}_{*,1}+B_3\vec{R}_{*,2}\Bigr)\\
&+\cdots\\
&+K\Bigl(\bigl(B_f\vec{f}_{h,N_t-1}+B_g\vec{g}_{h,N_t-1}\bigr)\Delta W_{N_t-1}+B_1\vec{E}_{*,N_t}+B_2\vec{E}_{*,N_t-1}+B_3\vec{R}_{*,N_t}\Bigr).
\end{aligned}
\end{equation}
In \eqref{k_is_m_1}, the right‑hand side involves known quantities (the initial state \(\Phi_0\) and \(\vec{E}_{*,0}\)) as well as unknown states at later times, while the left‑hand side \(\Phi_{N_t}\) is unknown.  Now define
\[\vec{E}_{*,s}=\begin{bmatrix}
(\vec{E}_{*,1})^{\top}&\cdots&(\vec{E}_{*,N_t})^{\top}
\end{bmatrix}^{\top}, \q
\vec{R}_{*,s}=\begin{bmatrix}
(\vec{R}_{*,1})^{\top}&\cdots&(\vec{R}_{*,N_t})^{\top}
\end{bmatrix}^{\top},\]
\[\xi = \Phi_{N_t} - U^{N_t}\Phi_0 - U^{N_t-1}K B_2\vec{E}_{*,0}.\] 
Then we obtain
\begin{equation}\label{xi}
S_f\vec{f}_{h,s}+S_g\vec{g}_{h,s}+S_1\vec{E}_{*,s}+S_r\vec{R}_{*,s}=\xi,
\end{equation}
where each matrix \(S_f, S_g, S_1, S_r\) contains \(N_t\) blocks and are given by
\begin{align*}
S_f &=
\begin{bmatrix}
U^{N_t-1}KB_f\Delta W_0 & U^{N_t-2}KB_f\Delta W_1 & \cdots & UKB_f\Delta W_{N_t-2} & KB_f\Delta W_{N_t-1}
\end{bmatrix},\\
S_g &=
\begin{bmatrix}
U^{N_t-1}KB_g\Delta W_0 & U^{N_t-2}KB_g\Delta W_1 & \cdots & UKB_g\Delta W_{N_t-2} & KB_g\Delta W_{N_t-1}
\end{bmatrix},\\
S_1 &=
\begin{bmatrix}
U^{N_t-2}S_0 & U^{N_t-3}S_0 & \cdots & U S_0 & S_0 & KB_1
\end{bmatrix},\\
S_r &=
\begin{bmatrix}
U^{N_t-1}KB_3 & U^{N_t-2}KB_3 & \cdots & UKB_3 & KB_3
\end{bmatrix},
\end{align*}
with the matrix \(S_0\) appearing in \(S_1\) defined as
$S_0 = K(UB_1 + B_2).$

   \vspace{1em}
\begin{algorithm}[H]
\caption{A fully discrete scheme for the stochastic Maxwell equation}
\textbf{Input:}
\begin{itemize}
\item Space Partition: \(x_i\in [a_i,b_i]\), and the number of grid cells, \(N_i\), \(i=1,2,3\).
			\item Time Partition: \(t\in(0,T)\), and the number of time steps, \(N_t\), \(\Delta t=\frac{T}{N_t}\).
			\item Initial Data: \(({\bf E}_0(x)\), \({\bf H}_0(x))\).
			\item Terminal Target: \(({\bf E}_T(x)\), \({\bf H}_T(x))\).
		\end{itemize}
		
		\textbf{Output:}
		\begin{itemize}
			\item 	Equation \eqref{div_condition}: \(\left\{\begin{aligned}
				&V\vec{f}_{h,s}=0,\\
				&P\vec{g}_{h,s}+Q\vec{g}_{*,s}=0.
			\end{aligned}\right.\)
			\item Equation \eqref{xi}: \(S_f\vec{f}_{h,s}+S_g\vec{g}_{h,s}+S_1\vec{E}_{*,s}+S_r\vec{R}_{*,s}=\xi\).
		\end{itemize}

\begin{enumerate}[leftmargin=0.5pt]
\item[1.] Discretize in space via central difference and in time via midpoint method to obtain the fully discrete scheme \eqref{fully_discrete}.
\begin{equation}
\left\{\begin{aligned}
\vec{E}_{h,n+1}-\tfrac{\Delta t}{2}A\vec{H}_{h,n+1}&=\vec{E}_{h,n}+\tfrac{\Delta t}{2}A\vec{H}_{h,n}+\vec{f}_{h,n}\Delta W_n,\\
-\tfrac{\Delta t}{2}F\vec{E}_{h,n+1}+\vec{H}_{h,n+1}&=\tfrac{\Delta t}{2}F\vec{E}_{h,n}+\vec{H}_{h,n}+\vec{g}_{h,n}\Delta W_n+\tfrac{\Delta t}{2}G\vec{E}_{*,n+1}+\tfrac{\Delta t}{2}G\vec{E}_{*,n}.
\end{aligned}\right.\notag
\end{equation}
\item[2.] %Considering 
Consider the divergence-free condition. We require \eqref{div_condition} 
\[\left\{\begin{aligned}
&V\vec{f}_{h,s}=0,\\
&P\vec{g}_{h,s}+Q\vec{g}_{*,s}=0
\end{aligned}\right.\]
to hold.

\item[3.] Incorporate all boundary points into the fully discrete scheme \eqref{fully_discrete}. Then based on the initial data and terminal target, we require \eqref{xi} 
\begin{equation}
S_f\vec{f}_{h,s}+S_g\vec{g}_{h,s}+S_1\vec{E}_{*,s}+S_r\vec{R}_{*,s}=\xi\notag
\end{equation}
to hold.
\end{enumerate}
\end{algorithm}

\subsubsection{Numerical simulation for The Unknown Control
}
In this subsection, we present  an optimization problem
	\begin{equation}
		\begin{split}
			J(\vec{f}_{h,s},\vec{g}_{h,s},\vec{g}_{*,s},\vec{E}_{*,s},\vec{R}_{*,s})
			\!=\!\tfrac{1}{2}((\vec{f}_{h,s})^{\top} \vec{f}_{h,s}\!+\!(\vec{g}_{h,s})^{\top} \vec{g}_{h,s}\!+\!(\vec{g}_{*,s})^{\top} \vec{g}_{*,s}+(\vec{E}_{*,s})^{\top} \vec{E}_{*,s}\!+\!(\vec{R}_{*,s})^{\top} \vec{R}_{*,s}),\notag
		\end{split}
	\end{equation}
	to find \(\vec{f}_{h,s},\vec{g}_{h,s},\vec{g}_{*,s},\vec{E}_{*,s},\vec{R}_{*,s}\) that minimizes \(J\).
	In detail, 
	we apply the Lagrange multiplier method by introducing 
    %make use of the Lagrange multiplier method, introducing 
    three Lagrange multipliers \(\lambda_1\), \(\lambda_2\), \(\lambda_3\), and defining %define 
    the following Lagrangian function
	\begin{equation}\label{Lagrange_function}
		\begin{split}
			L(&\vec{f}_{h,s},\vec{g}_{h,s},\vec{g}_{*,s},\vec{E}_{*,s},\vec{R}_{*,s},\lambda_1,\lambda_2,\lambda_3)
			=J(\vec{f}_{h,s},\vec{g}_{h,s},\vec{g}_{*,s},\vec{E}_{*,s},\vec{R}_{*,s})\\
			&+\lambda_1(0-V\vec{f}_{h,s})+\lambda_2(0-P\vec{g}_{h,s}-Q\vec{g}_{*,s})+\lambda_3(\xi-S_f\vec{f}_{h,s}-S_g\vec{g}_{h,s}-S_1\vec{E}_{*,s}-S_r\vec{R}_{*,s}).
		\end{split}
	\end{equation}
By taking the partial derivatives of \(L\) with respect to \(\vec{f}_{h,s},\vec{g}_{h,s},\vec{g}_{*,s},\vec{E}_{*,s},\vec{R}_{*,s}\) and \(\lambda_1,\lambda_2,\lambda_3\), respectively, and setting  them equal to zero, we get %obtain
\begin{equation}\label{optinal_solutions}
\begin{aligned}
&\vec{f}_{h,s}=V^{\top}\lambda_1+S_f^{\top}\lambda_3,\quad \vec{g}_{h,s}=P^{\top}\lambda_2+S_g^{\top}\lambda_3,\\
&\vec{g}_{*,s}=Q^{\top}\lambda_2,\quad \vec{E}_{*,s}=S_1^{\top}\lambda_3,\quad \vec{R}_{*,s}=S_r^{\top}\lambda_3,
\end{aligned}
\end{equation}
\begin{equation}
VV^{\top}\lambda_1+VS_f^{\top}\lambda_3=0,\notag
\end{equation}
\begin{equation}
(PP^{\top}+QQ^{\top})\lambda_2+PS_g^{\top}\lambda_3=0,\notag
\end{equation}
\begin{equation}\label{unknown_lambda_3_}
		S_fV^{\top}\lambda_1+S_gP^{\top}\lambda_2+(S_fS_f^{\top}+S_gS_g^{\top}+S_1S_1^{\top}+S_rS_r^{\top})\lambda_3=\xi.
	\end{equation}
	It can be proved that \(VV^{\top}\) and \(PP^{\top}+QQ^{\top}\) are invertible. As a result,
	\begin{equation}\label{lambda_1}
		\lambda_1=-(VV^{\top})^{-1}VS_f^{\top}\lambda_3,
	\end{equation}
	\begin{equation}\label{lambda_2}
		\lambda_2=-(PP^{\top}+QQ^{\top})^{-1}PS_g^{\top}\lambda_3.
	\end{equation}
	Substituting \(\lambda_1\) and \(\lambda_2\) into \eqref{unknown_lambda_3_}, we obtain
	\begin{equation}\label{lambda_3}
    \begin{aligned}
\lambda_3=&(S_fS_f^{\top}+S_gS_g^{\top}+S_1S_1^{\top}+S_rS_r^{\top}
-S_fV^{\top}(VV^{\top})^{-1}VS_f^{\top}\\
&-S_gP^{\top}(PP^{\top}+QQ^{\top})^{-1}PS_g^{\top})^{-1}\xi.   
\end{aligned}
\end{equation}
Combining \(\lambda_1\), \(\lambda_2\) and \(\lambda_3\) with \eqref{optinal_solutions}, we arrive at the values of \(\vec{f}_{h,s},\vec{g}_{h,s},\vec{g}_{*,s},\vec{E}_{*,s},\vec{R}_{*,s}\).

    \vspace{1em}
	\begin{algorithm}[H]
\caption{Numerical algorithm for simulating the unknown control}
		\textbf{Input:}
		\begin{itemize}
			\item 	Equation \eqref{div_condition}: \(\left\{\begin{aligned}
				&V\vec{f}_{h,s}=0,\\
				&P\vec{g}_{h,s}+Q\vec{g}_{*,s}=0.
			\end{aligned}\right.\)
			\item Equation \eqref{xi}: \(S_f\vec{f}_{h,s}+S_g\vec{g}_{h,s}+S_1\vec{E}_{*,s}+S_r\vec{R}_{*,s}=\xi\).
		\end{itemize}
%		\begin{itemize}
%			\item Equations \eqref{div_condition} and \eqref{xi}.
%		\end{itemize}
		
		\textbf{Output:}
		\begin{itemize}
			\item Optimal Control \((f,g,u)\).
		\end{itemize}
\begin{enumerate}
[leftmargin=0.5pt]
    \item[1.] Define the Lagrange function \eqref{Lagrange_function} with three multipliers \(\lambda_1\), \(\lambda_2\), \(\lambda_3\):
		\begin{equation}
			\begin{split}
				L&=\tfrac{1}{2}((\vec{f}_{h,s})^{\top} \vec{f}_{h,s}+(\vec{g}_{h,s})^{\top} \vec{g}_{h,s}+(\vec{g}_{*,s})^{\top} \vec{g}_{*,s}+(\vec{E}_{*,s})^{\top} \vec{E}_{*,s}+(\vec{R}_{*,s})^{\top} \vec{R}_{*,s})\\
				&+\lambda_1(0-V\vec{f}_{h,s})+\lambda_2(0-P\vec{g}_{h,s}-Q\vec{g}_{*,s})+\lambda_3(\xi-S_f\vec{f}_{h,s}-S_g\vec{g}_{h,s}-S_1\vec{E}_{*,s}-S_r\vec{R}_{*,s}).\notag
			\end{split}
		\end{equation}

    \item[2.] Take the partial derivatives of \(L\) with respect to every variables and set them equal to zero. Then we derive
		\[\begin{aligned}
			\vec{f}_{h,s}=V^{\top}\lambda_1+S_f^{\top}\lambda_3,&&\vec{g}_{h,s}=P^{\top}\lambda_2+S_g^{\top}\lambda_3,&&\vec{E}_{*,s}=S_1^{\top}\lambda_3,
		\end{aligned}\]
		where \(\lambda_1\), \(\lambda_2\), \(\lambda_3\) can be obtained from \eqref{lambda_1}, \eqref{lambda_2}, \eqref{lambda_3}, respectively.
\item[3.] Compute the control \((f,g,u)\).
\begin{itemize}
\item the value of \(f\) can be obtained from \( \vec{f}_{h,s}\).
\item the value of \(g\) can be obtained from \( \vec{g}_{h,s}\).
\item the value of \(u\) can be obtained from \( \vec{E}_{*,s}\).
\begin{itemize}
				%				\item \(E_*=\begin{bmatrix}
					%					(\vec{E}_{1_2})^{\top}&(\vec{E}_{1_3})^{\top}&(\vec{E}_{2_1})^{\top}&(\vec{E}_{2_3})^{\top}&(E_{3_1})^{\top}&(\vec{E}_{3_2})^{\top}
					%				\end{bmatrix}^{\top}\)
\item \(u|_{x_1=a_1}= (0,-E_3,E_2)\), \(u|_{x_1=b_1}=(0,E_3,-E_2)\) 
can be obtained from \(\vec{E}_{3_1}\) and \(\vec{E}_{2_1}\),
\item \(u|_{x_2=a_2}= (E_3,0,-E_1)\), \(u|_{x_2=b_2}= (-E_3,0,E_1)\) 
can be obtained from \(\vec{E}_{3_2}\) and \(\vec{E}_{1_2}\),
\item \(u|_{x_3=a_3}= (-E_2,E_1,0)\), \(u|_{x_3=b_3}= (E_2,-E_1,0)\) 
can be obtained from \(\vec{E}_{2_3}\) and \(\vec{E}_{1_3}\).
\end{itemize}
\end{itemize}
\end{enumerate} 
\end{algorithm}

\subsection{Numerical Experiments.}
\captionsetup[subfigure]{font=tiny,labelsep=none,skip=2pt}
	
	\begin{figure}[htbp]
		\centering
		% 第一列 E1
		\begin{minipage}[t]{0.28\textwidth}
			\centering
			\subcaptionbox{}{\includegraphics[width=\linewidth]{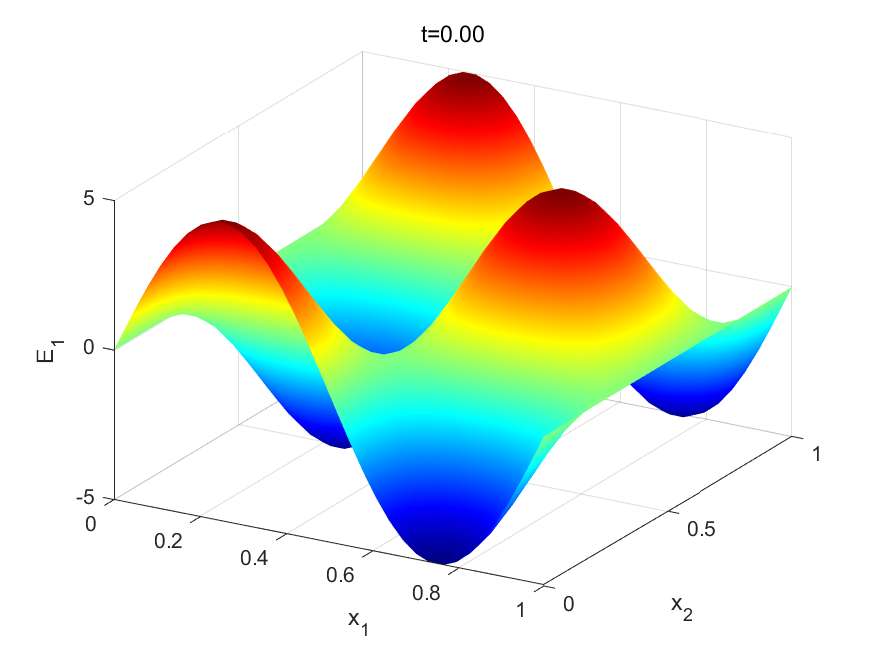}}\\
			\subcaptionbox{}{\includegraphics[width=\linewidth]{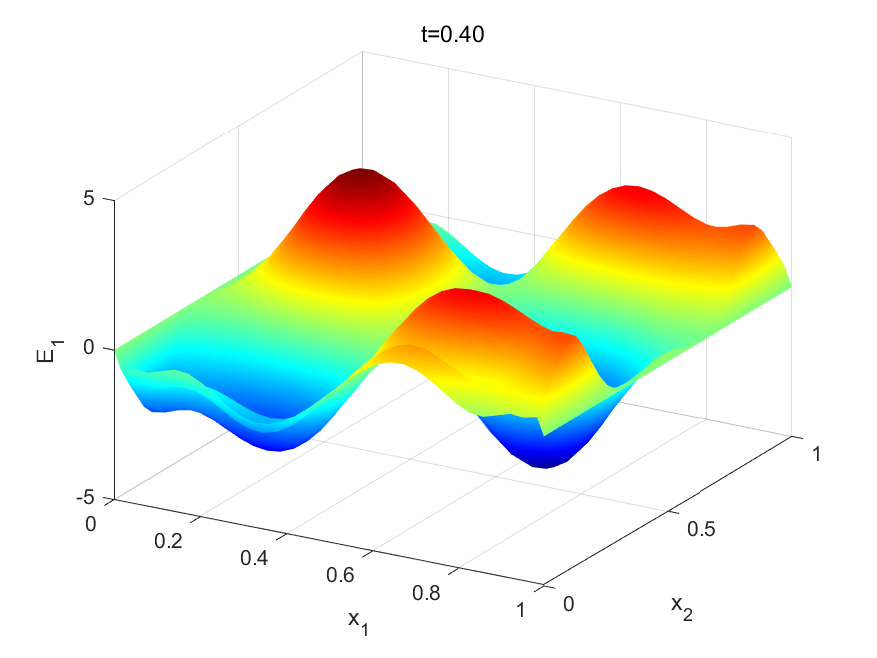}}\\
			\subcaptionbox{}{\includegraphics[width=\linewidth]{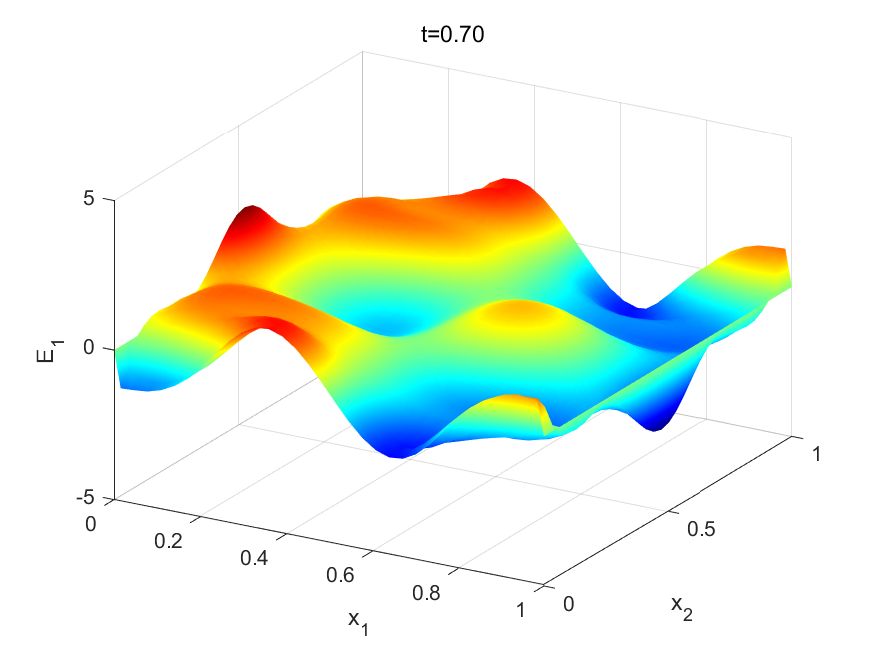}}\\
			\subcaptionbox{}{\includegraphics[width=\linewidth]{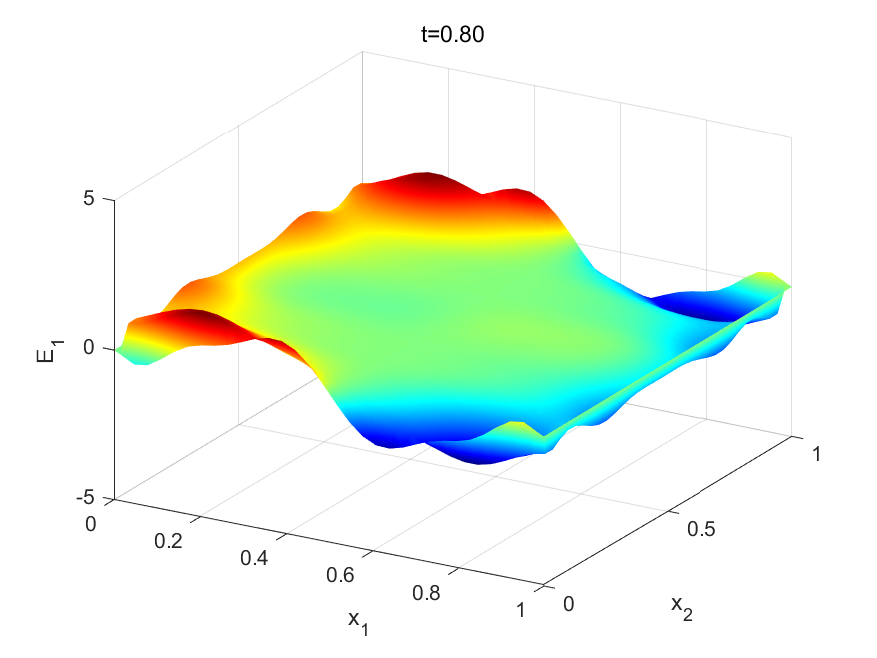}}\\
			\subcaptionbox{}{\includegraphics[width=\linewidth]{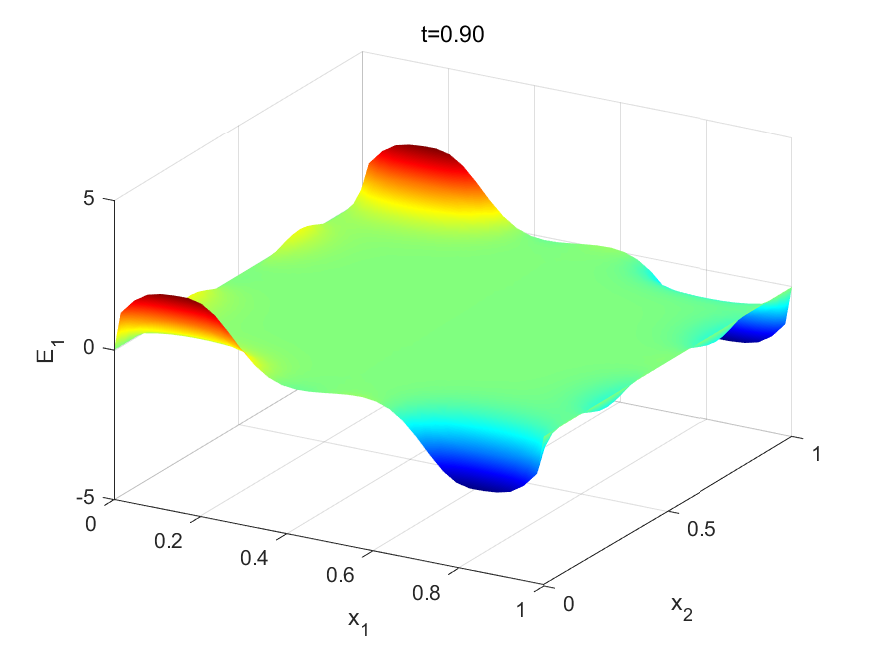}}\\
			\subcaptionbox{}{\includegraphics[width=\linewidth]{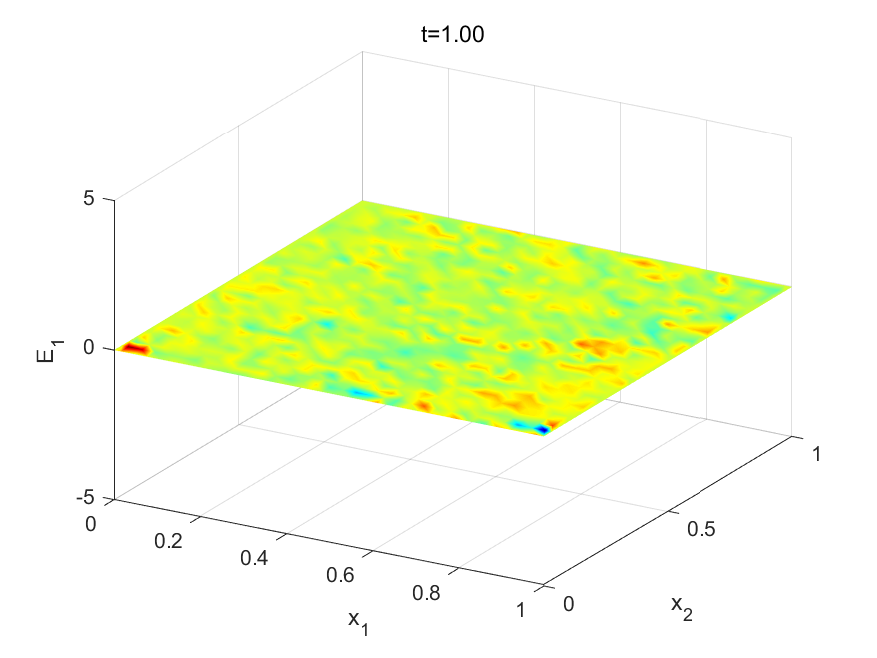}}
		\end{minipage}
		\hfill
		% 第二列 E2
		\begin{minipage}[t]{0.28\textwidth}
			\centering
			\subcaptionbox{}{\includegraphics[width=\linewidth]{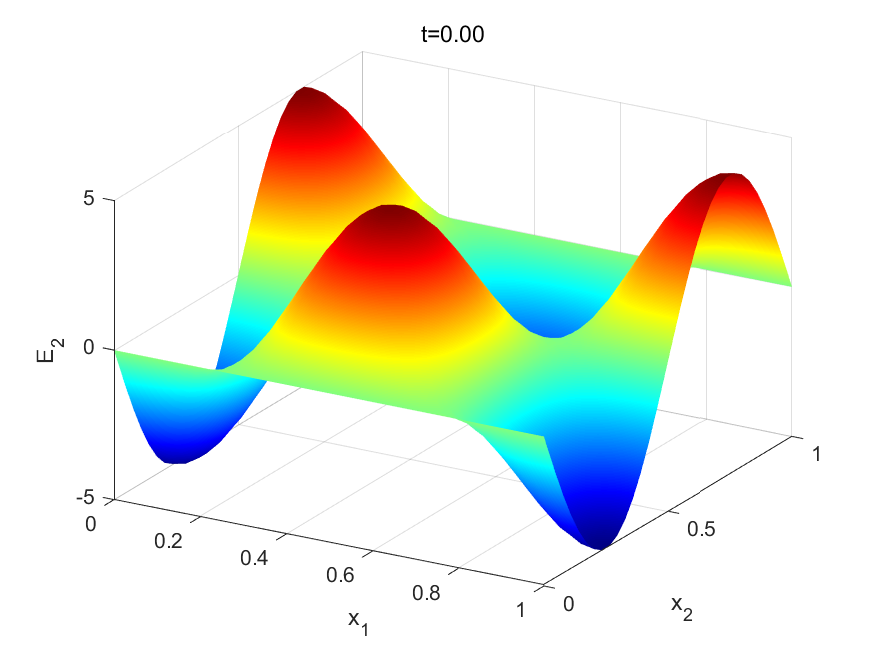}}\\
			\subcaptionbox{}{\includegraphics[width=\linewidth]{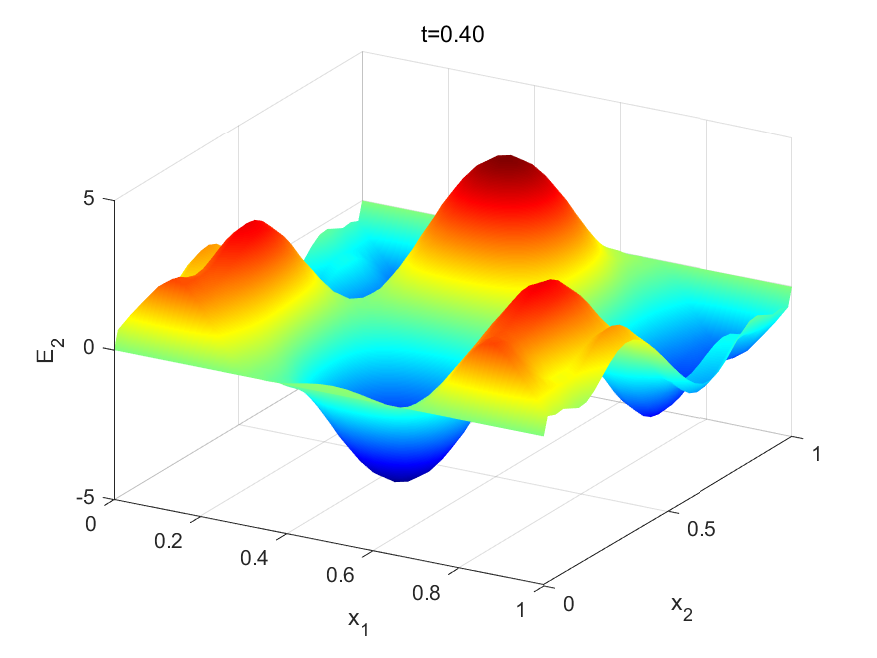}}\\
			\subcaptionbox{}{\includegraphics[width=\linewidth]{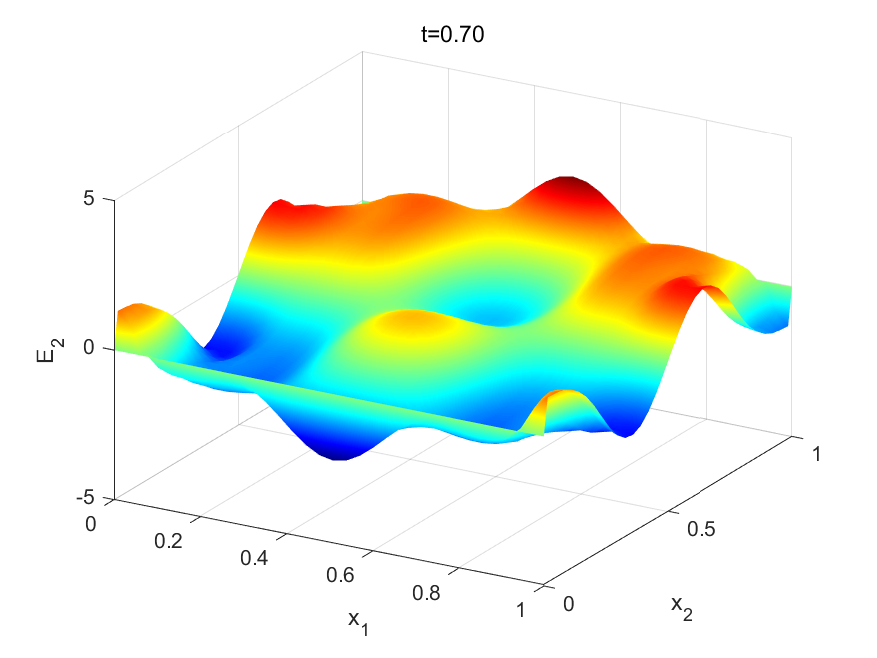}}\\
			\subcaptionbox{}{\includegraphics[width=\linewidth]{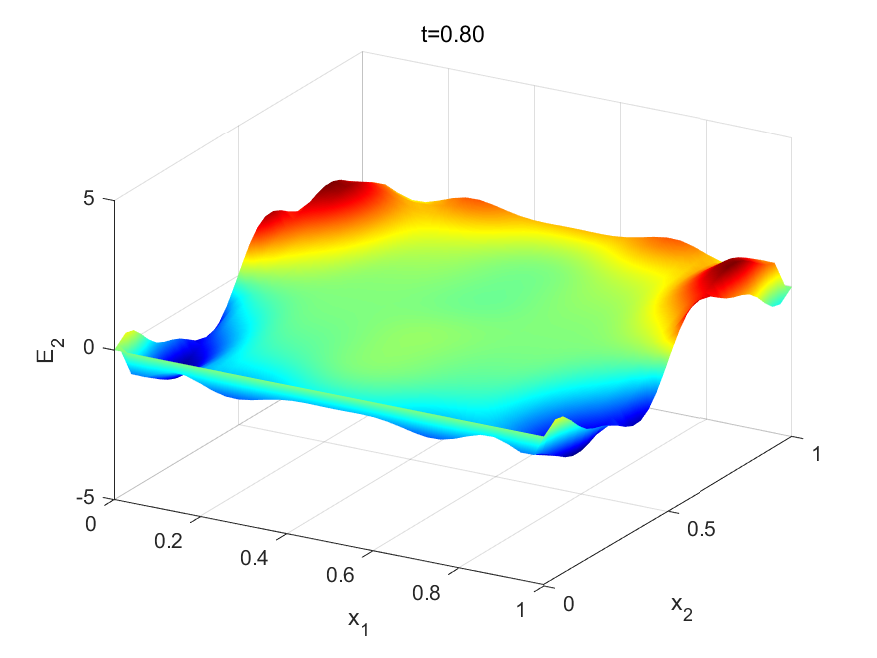}}\\
			\subcaptionbox{}{\includegraphics[width=\linewidth]{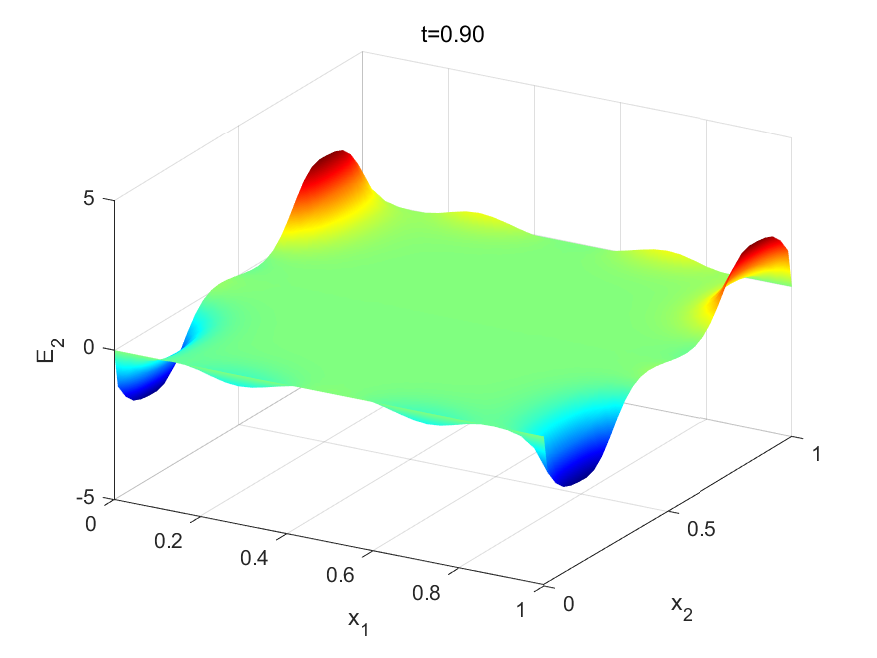}}\\
			\subcaptionbox{}{\includegraphics[width=\linewidth]{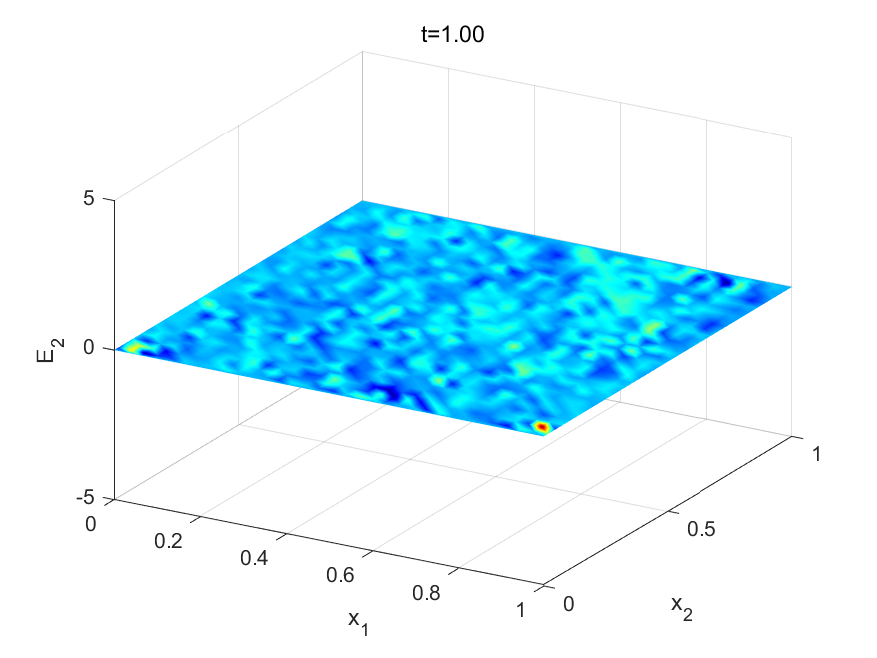}}
		\end{minipage}
		\hfill
		% 第三列 H3
		\begin{minipage}[t]{0.28\textwidth}
			\centering
			\subcaptionbox{}{\includegraphics[width=\linewidth]{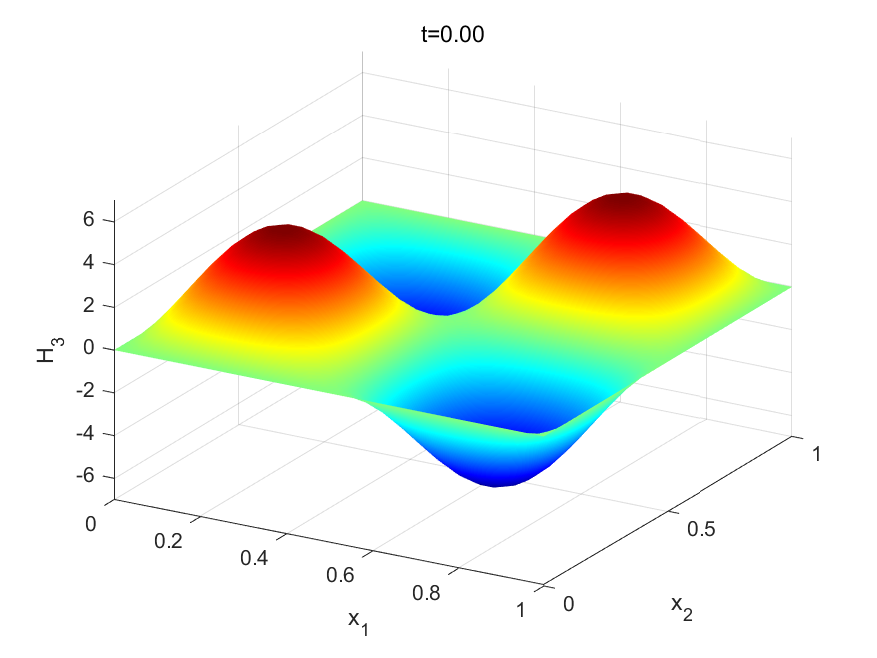}}\\
			\subcaptionbox{}{\includegraphics[width=\linewidth]{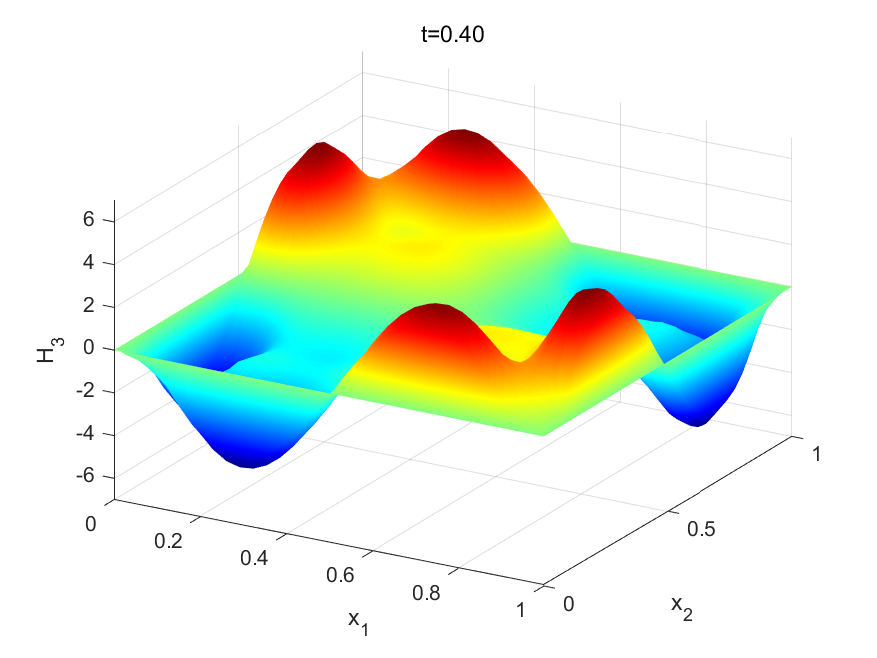}}\\
			\subcaptionbox{}{\includegraphics[width=\linewidth]{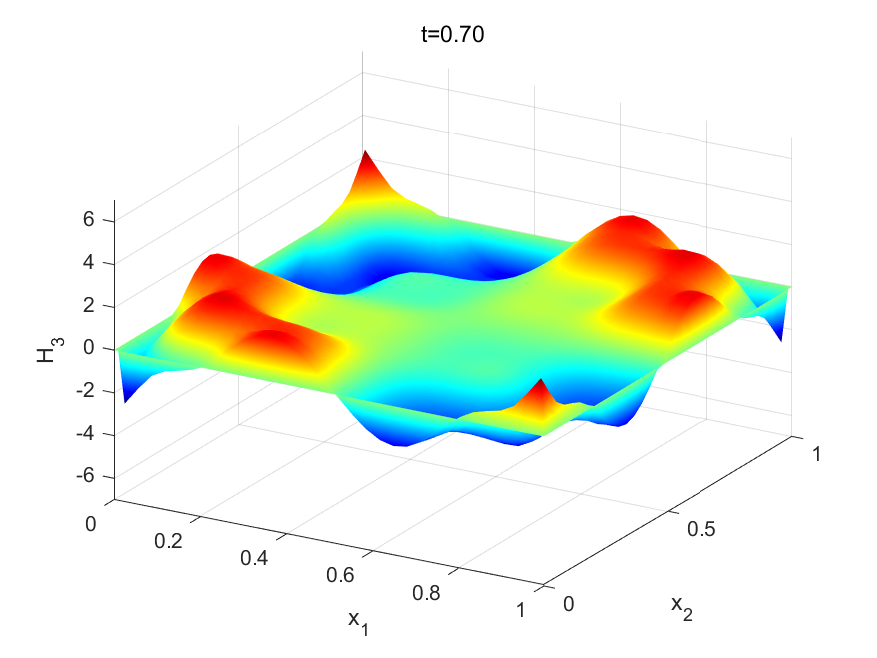}}\\
			\subcaptionbox{}{\includegraphics[width=\linewidth]{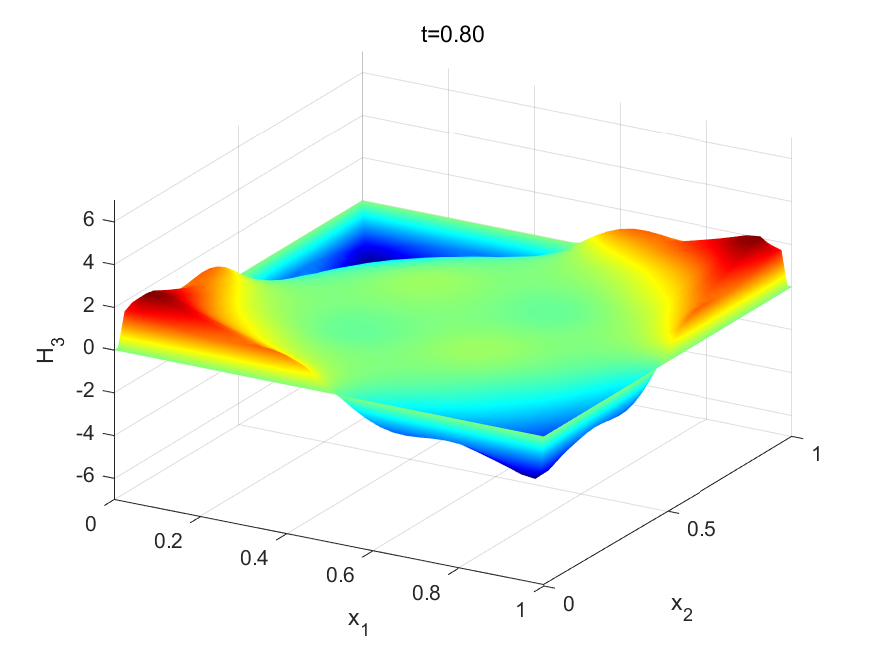}}\\
			\subcaptionbox{}{\includegraphics[width=\linewidth]{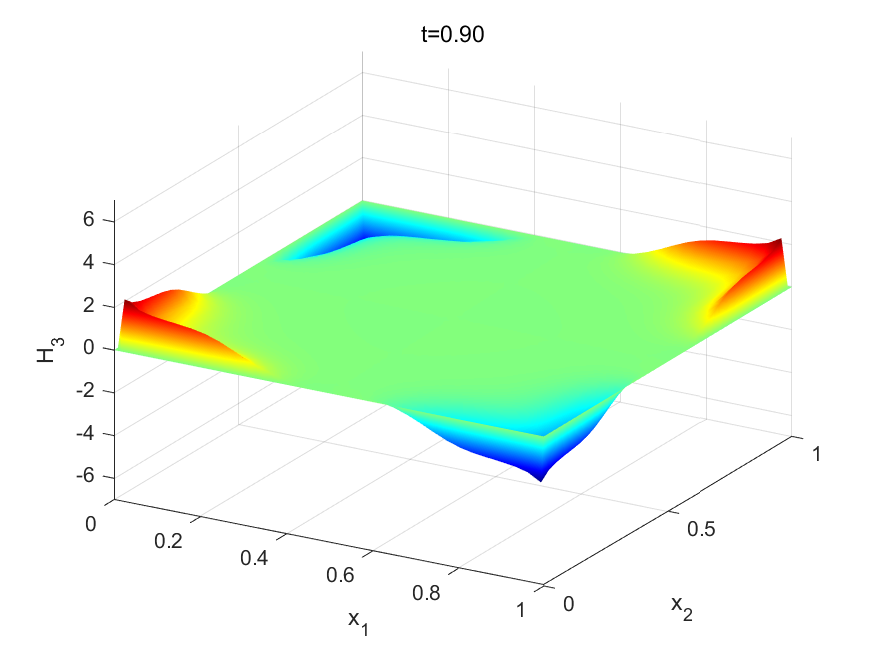}}\\
			\subcaptionbox{}{\includegraphics[width=\linewidth]{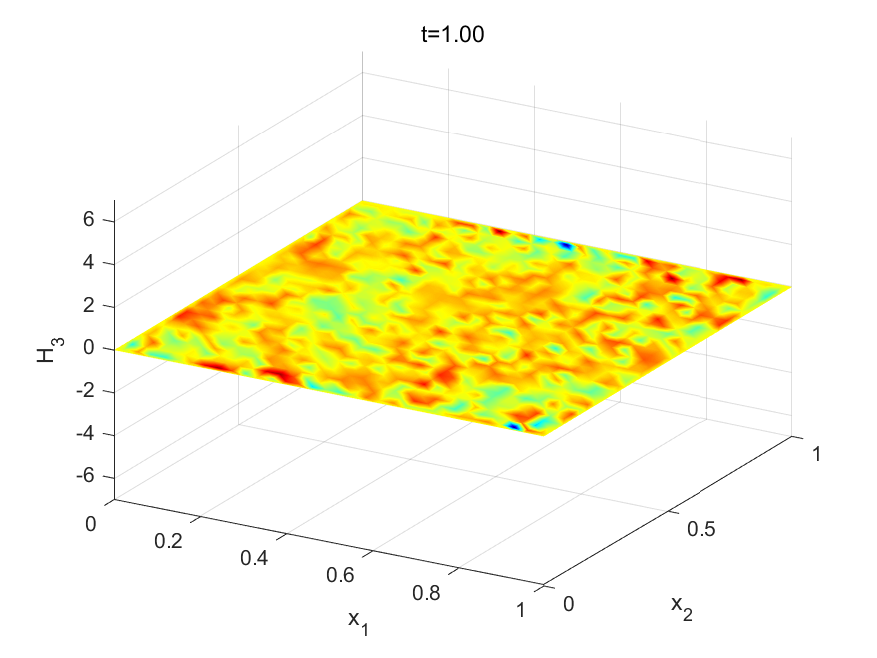}}
		\end{minipage}
		
		\caption{Figures (A)-(F), (G)-(L), and (M)-(R) show the evolution of \(E_{1}\), \(E_{2}\), and \(H_{3}\) at times \(t=0,\,0.4,\,0.7,\,0.8,\,0.9,\,1.0\), respectively.}
		\label{fig:combined}
	\end{figure}

%The three-dimensional stochastic Maxwell equations can be reduced to the two-dimensional \(TE_z\) mode or the \(TM_z\) mode. 
In the numerical experiments, we consider the two-dimensional \(TE_z\) mode, in which the electric and magnetic fields are given by \( {\bf{E}}=(E_1,E_2,0) \) and \({\bf{H}}=(0,0,H_3)\) with each component depending on the spatial variables \(x_1\), \(x_2\) and the time variable %time 
\(t\). The equation of \(TE_z\) mode reads
\begin{equation}\label{TE_z_equation}
\left\{
\begin{aligned}
& \mbox{d}E_1= \tfrac{\partial H_3}{\partial x_2}\mbox{d}t + f_1 \, \mbox{d}W(t) && \text{in } Q, \\
& \mbox{d}E_2 =-\tfrac{\partial H_3}{\partial x_1}\mbox{d}t + f_2 \, \mbox{d}W(t) && \text{in } Q, \\
& \mbox{d}H_3 = ( \tfrac{\partial E_1}{\partial x_2} - \tfrac{\partial E_2}{\partial x_1})\mbox{d}t + g_3 \, \mbox{d}W(t) && \text{in } Q, \\
& \tfrac{\partial E_1}{\partial x_1}+\tfrac{\partial E_2}{\partial x_2}=0 && \text{in } Q,\\
& \mathbf{E} \times \nu = u && \text{on } \Gamma \times (0,T), \\
& \mathbf{E}(x,0) = \mathbf{E}_0(x), \quad \mathbf{H}(x,0) = \mathbf{H}_0(x) && \text{in } G,
\end{aligned}
\right.
\end{equation}
where the domain \( G = (0,1)\times (0,1) \) and the terminal time \( T = 1 \), with spatial step sizes \( h_1 = h_2 = \tfrac{1}{2^5} \) and time step size \( \Delta t = \tfrac{1}{10}\). 
We take the following initial data 
	\begin{equation}
		\begin{aligned}
			&E_{1}(0)=5\sin(2\pi x_1)\cos(2\pi x_2),\\&E_{2}(0)=-5\cos(2\pi x_1)\sin(2\pi x_2),\\&H_{3}(0)=5\sin(2\pi x_1)\sin(2\pi x_2),\notag
		\end{aligned}
	\end{equation}
and take \(({\bf E}_T(x),{\bf H}_T(x))=(0,0)\) as the terminal target. By implementing the numerical algorithm described above, we obtain controls \( (f, g, u) \) such that \(\bf E\) and \(\bf H\) satisfy the terminal target at time \( T = 1 \). 
Substituting the numerical values of the control into the full discretization yields the results presented in Figure \ref{fig:combined}. 
It can be found from Figure~\ref{fig:combined} that the spatio-temporal evolution of the components \(E_1\), \(E_2\), and \(H_3\) at six different time instants. 
The evolution reveals that the solution gradually converges to the terminal target  starting around \(t = 0.7\) under the given control. Notably, the interior nodes of the domain reach the terminal target first. 
Then the state propagates from the interior toward the boundary nodes until the entire system eventually approaches the terminal target.
	
%Figure \ref{fig:combined} presents spatiotemporal evolution profiles of the components \(E_{1},\,E_{2},\,H_{3}\) at six different time instants. It can be observed from the evolution that the numerical solution begins to gradually converge to the terminal target at \(t=0.7\). We can see that the interior nodes of the domain first reach the terminal target, and then the state transitions from the interior to the boundary nodes, until the whole system eventually approaches the terminal target. 

\begin{table}[htbp]
\centering
\caption{\(\mathrm {MSE}\) of numerical approximations for different variables under different types of control.}
\label{tab:convergece}
\begin{tabular}{c c c c}
\toprule
Type of control & \(\mathrm {MSE}_{E_1}\) & \(\mathrm {MSE}_{E_2}\)& \(\mathrm {MSE}_{H_3}\)\\
\midrule
\((\vec{f}_1,\vec{f}_2,\vec{g}_3,\vec{E}_{1_2},\vec{E}_{2_1})\) & \(7.70\times 10^{-23}\) & \(7.08\times 10^{-23}\) & \(2.46\times 10^{-23}\) \\
\midrule
\((\vec{f}_2,\vec{g}_3,\vec{E}_{1_2},\vec{E}_{2_1})\) & \(4.75\times 10^1\) & \(6.10\times 10^1\) & \(6.32\times 10^1\) \\
\((\vec{f}_1,\vec{g}_3,\vec{E}_{1_2},\vec{E}_{2_1})\) & \(6.10\times 10^1\) & \(4.75\times 10^1\) & \(6.32\times 10^1\) \\
\((\vec{g}_3,\vec{E}_{1_2},\vec{E}_{2_1})\)& \(1.03\times 10^2\) & \(1.03\times 10^2\) & \(1.84\times 10^2\) \\
\midrule
				\((\vec{f}_1,\vec{f}_2,\vec{E}_{1_2},\vec{E}_{2_1})\) & \(3.31\times 10^1\) & \(3.31\times 10^1\) & \(1.22\times 10^2\) \\
				\midrule
				\((\vec{f}_1,\vec{f}_2,\vec{g}_3,\vec{E}_{2_1})\)&\(6.75\times 10^2\)&\(2.43\times 10^2\)&\(3.21\times 10^2\)\\
				\((\vec{f}_1,\vec{f}_2,\vec{g}_3,\vec{E}_{1_2})\)&\(2.43\times 10^2\)&\(6.75\times 10^2\)&\(3.21\times 10^2\)\\
				\((\vec{f}_1,\vec{f}_2,\vec{g}_3)\)&\(8.05\times 10^2\)&\(8.05\times 10^2\)&\(1.12\times 10^3\)\\
				\bottomrule
			\end{tabular}
		\end{table}

        To evaluate how well the numerical solution approximates the terminal target, we adopt the mean squared error (\(\mathrm{MSE}\)) as a metric.  The \(\mathrm{MSE}\) for \(E_1\) is defined as
\[
\mathrm{MSE}_{E_1} = \sum_{i=0}^{N_1-1}\sum_{j=0}^{N_2} \bigl( E_{1,N_t}^{i+\frac{1}{2},j} - E_1^{i+\frac{1}{2},j}(T) \bigr)^2,
\]
where \(E_{1,N_t}^{i+\frac{1}{2},j}\) denotes the approximation of \(E_1^{i+\frac{1}{2},j}(T)\).  The quantities \(\mathrm{MSE}_{E_2}\) and \(\mathrm{MSE}_{H_3}\) are defined analogously. 
Table~\ref{tab:convergece} presents the MSE of numerical approximations obtained under different type of control. 
It can be verified that the stochastic Maxwell system is controllable when all three controls \((f, g, u)\) are present, but becomes uncontrollable if any one of them is missing.  
These are coincide with the theoretical results.

\end{document}